\documentclass{amsart}[10pt,fullpage]

\usepackage{amsfonts, amstext, amscd, amsmath,amssymb,amsthm}
\usepackage[numbers,sort]{natbib}
\usepackage{multirow}
\usepackage{xcolor,mathrsfs,graphicx}
\usepackage{tikz}
\usetikzlibrary{patterns}

\newtheorem{thm}{Theorem}[section]
\newtheorem{lem}[thm]{Lemma}
\newtheorem{prop}[thm]{Proposition}

\theoremstyle{definition}
\newtheorem{exmp}{Example}
\newtheorem{ques}{Question}
\newtheorem{defn}[thm]{Definition}
\newtheorem{rem}[thm]{Remark}
\newtheorem*{ack}{Acknowledgments}

\numberwithin{equation}{section}
\input{xy}
\xyoption{all}

\newcommand{\comment}[1]{}

\newcommand{\bd}{\partial}

\newcommand{\ints}{\mathbb Z}
\newcommand{\rls}{\mathbb R}

\newcommand{\rat}{\mathbb Q}
\newcommand{\bb}[1]{\mathbb{#1}}
\newcommand{\cl}[1]{\mathcal{#1}}
\newcommand{\scr}[1]{\mathscr{#1}}

\newcommand{\wt}[1]{\widetilde{#1}}

\newcommand{\im}{\text{Im }}

\newcommand{\al}[1]{\begin{align*}#1\end{align*}}
\newcommand{\en}[1]{\begin{enumerate}#1\end{enumerate}}

\newcommand{\set}[1]{\left\{#1\right\}}
\newcommand{\setn}[2]{\left\{#1\mid #2\right\}}

\newcommand{\abs}[1]{\left\vert#1\right\vert}

\newcommand{\tb}{\text{tb}}
\newcommand{\rot}{\text{rot}}

\begin{document}

\title[]{A polynomial invariant for links in lens spaces}
\author[]{Christopher R. Cornwell}
\address[]{Department of Mathematics, Michigan State University, East
Lansing, MI 48824}
\email[]{cornwell@math.msu.edu}

\begin{abstract}
We prove the existence of a polynomial invariant that satisfies the HOMFLY skein relation for links in a lens space. In the process we also develop a skein theory of toroidal grid diagrams in a lens space.
\end{abstract}

\maketitle

 \section{Introduction}
 \label{sec:intro}
  
 The HOMFLY polynomial is an invariant of links in $S^3$, introduced in \cite{FHLMOY} (see also \cite {PT}), and may be defined via a linear skein relation along with a choice of normalization on the unknot. As observed by Ocneanu, the HOMFLY polynomial can also be expressed using a trace of representations into Hecke algebras (see also \cite{J1}). 
 
 Original constructions of the HOMFLY polynomial are difficult to replicate in manifolds other than $S^3$ as they rely on planar link projections. However, Birman and Lin show in \cite{BirL} that the HOMFLY and Kauffman polynomials are generating functions for particular sequences of Vassiliev (or finite-type) link invariants (see also \cite{DBN1},\cite{DBN2}, and \cite{Kont}). In \cite{Ka} and \cite{KL}, Vassiliev invariants in other 3-manifolds were studied from an intrinsic point of view. In particular let $M$ be a rational homology 3-sphere that is either atoroidal or Seifert-fibered with orientable orbit space. Fix a collection $\scr{TL}$ of ``trivial links'' in $M$ having one representative from each homotopy class of links. Let $\scr{TL}^*\subset\scr{TL}$ denote the collection of trivial links that have no nullhomotopic components. Kalfagianni and Lin show \cite{KL} that, given a choice of value $J_M(T)$ for each $T\in\scr{TL}^*$, and a choice $J_M(U)$ for the standard unknot $U$, there exists a unique power-series valued invariant $J_M$ with coefficients that are Vassiliev invariants, satisfying the HOMFLY skein relation. 

Given $M$ as above, it is unknown whether $J_M$ can be made to take values that are Laurent polynomials. In this consideration, Kalfagianni asked the following question \cite{Ka2} which encapsulates the primary difficulty of the problem:
\begin{ques} Is there a choice of $\scr{TL}^*$ such that every link $L\subset M$ can be reduced to disjoint unions of unlinks and elements in $\scr{TL}^*$ by a series of finitely many skein moves?
\label{ques1}
\end{ques}

In this article we use \emph{toroidal grid diagrams} to study links in lens spaces and develop a skein theory of such grid diagrams. With this theory in hand, we then define for each lens space $L(p,q)$ a collection of trivial links $\scr{TL}$ and show, via a complexity $\psi$ on grid diagrams, that this choice provides a positive answer to Question \ref{ques1} in lens spaces. As a result, the invariant $J_M$ of Kalfagianni and Lin provides a HOMFLY polynomial when $M$ is a lens space.

Recall that in the setting of $S^3$, to show that every link can be reduced by skein moves to a disjoint union of unlinks, we consider a planar projection of the link and induct on the complexity that uses the crossing number and unlinking number of the projection. Our complexity $\psi$ is inspired by this, but works with skein theory on grid diagrams. In particular, part of the definition of $\psi$ involves a function $scr$ that decreases under resolutions. However, it does not count the analogues of crossings in our skein theory, but rather objects related to these crossings. 

Throughout the paper we let $p,q$ be coprime integers with $0\le\abs q<p$. The lens space which results as $-\frac pq$ surgery on the unknot in $S^3$ is denoted by $L(p,q)$ (note that we do not consider the case $S^1\times S^2$). We discuss links in $L(p,q)$ through the tool of toroidal grid diagrams, as developed by Baker, Grigsby, and Hedden \cite{BGH,BG}. Though we do not treat the subject here, such diagrams correspond to projections of Legendrian links in $L(p,q)$ equipped with a universally tight contact structure. Moreover, the polynomial given in this paper is related to the contact geometry of these Legendrian links (see \cite{C2}).

We construct a collection of grid diagrams which we call trivial link diagrams. Associated to each trivial link diagram is a link in $L(p,q)$. Let a \emph{trivial link} be a link that is isotopic to this link. Our construction ensures that in any homotopy class of links, there is exactly one trivial link. The main result of the paper is stated as follows. 

\begin{thm} Let $\scr L$ be the set of isotopy classes of links in $L(p,q)$ and let $\scr{TL}\subset\scr L$ denote the set of isotopy classes of trivial links. Define $\scr{TL}^*\subset\scr{TL}$ to be those trivial links with no nullhomotopic components. Let $U$ be the isotopy class of the standard unknot, a local knot in $L(p,q)$ that bounds an embedded disk. Suppose we are given a value $J_{p,q}(T)\in\ints[a^{\pm1},z^{\pm1}]$ for every $T\in\scr{TL}^*$. Then there is a unique map $J_{p,q}:\scr{L}\to\ints[a^{\pm1},z^{\pm1}]$ such that
	\en{
	\item[(i)] $J_{p,q}$ satisfies the skein relation
		\begin{equation*}
		a^{-p}J_{p,q}(L_+)-a^pJ_{p,q}(L_-)=zJ_{p,q}(L_0).
		\end{equation*}
	\item[(ii)] $J_{p,q}(U)=a^{-p+1}$.
	\item[(iii)] $J_{p,q}\left(U\coprod L\right)=\frac{a^{-p}-a^p}{z}J_{p,q}(L)$.
	}
\label{mainThm}
\end{thm}

As usual, the links $L_+, L_-,$ and $L_0$ differ only in a small neighborhood. The exact construction of these links in $L(p,q)$ is made clear in the subsequent text and their construction is consistent with that in \cite{KL}. 

\begin{rem} In a large class of rational homology spheres, Kalfagianni has found a power series valued invariant of framed links that satisfies the Kauffman skein relation~\cite{Ka1}. The ideas in the proof of Theorem \ref{mainThm} should also be capable of showing that this invariant provides a Kauffman polynomial for links in $L(p,q)$.
\end{rem}
\begin{rem} The polynomial given by Theorem \ref{mainThm} is used in ~\cite{C2} to prove an analogue of the Franks-Williams-Morton inequality in $L(p,q)$ with a universally tight contact structure. This inequality exhibits a degree of the HOMFLY polynomial as an upper bound on the maximal self-linking number.
\end{rem}

We organize the paper as follows: In Section \ref{sec:prelim} we set notation, and will review constructions and results from \cite{KL} and \cite{BG,BGH}. In Section \ref{sec:hty} we prove some homotopy results and develop the skein theory on grid diagrams for links in lens spaces. In Section \ref{sec:HOMFLY} we prove our main results. We define a collection of trivial links and show that there is exactly one trivial link in each free homotopy class of links. Secondly, we show using a complexity function that if the power-series invariant is Laurent polynomial valued on each trivial link, then the same is true for any link. In Section \ref{sec:comput} we calculate the polynomial in some examples.

\begin{ack}The author would like to thank Eli Grigsby and Matt Hedden for their helpful conversations and input. He also thanks the referee for a careful reading and many helpful suggestions that helped to clarify the exposition. He also thanks his advisor Effie Kalfagianni for introducing him to these problems, for her expertise, and for many helpful discussions. This research was supported in part by NSF--RTG grant DMS-0353717, NSF grant DMS-0805942, and by a Herbert T. Graham scholarship.
\end{ack}

\section{Preliminaries}
\label{sec:prelim}

In this section we review constructions and results that will be important for the results we prove in sections \ref{sec:hty} and \ref{sec:HOMFLY}. The review is divided into two parts. The first part is devoted to 3-manifold techniques and the power-series invariant found in \cite{KL}, and in the second part we review the knot theory of grid diagrams in lens spaces, following \cite{BG},\cite{BGH}.

\subsection{Power series invariants in rational homology 3-spheres}
\label{subsec:psInvt}
Consider piecewise-linear links in a 3-manifold $M$ (the topic could also be treated in the smooth category). Let $P$ be a disjoint union of oriented circles. A piecewise-linear map $L:P\to M$ is an \emph{$n$-singular link} if it has exactly $n$ transverse double points. Equivalence of $n$-singular links is given by ambient isotopy in $M$ provided the double points remain transverse throughout the isotopy. We often write $L$ for the image $L(P)$. Note that a 0-singular link is a link in $M$.

For any $n$-singular link $L$, $P$ can be given the structure of a simplicial complex so that, for each double point $x$, there are distinct 1-simplexes $\sigma_1,\sigma_2\subset P$ with the interior of $\sigma_i$ containing one of the points in $L^{-1}(x)$. We may assume that $L(\sigma_1)\cup L(\sigma_2)$ is contained in an embedded disk $D\subset M$ with $L(\bd\sigma_1)\cup L(\bd\sigma_2)\subset\bd D$. Moreover, since $\sigma_i$ inherits an orientation from that of $P$, we may refer to the initial and terminal points of $\sigma_i$.

Let $x\in M$ be a double point of $L$. Subdividing the simplicial structure on $P$ if necessary, there is a small ball neighborhood $B\subset M$ of $x$ such that $D$ is properly embedded in $B$ and $L\cap B=L(\sigma_1)\cup L(\sigma_2)$. Define $a_1$ and $a_2$ to be two simple arcs in distinct components of $\bd B\setminus\bd D$ going from the initial point to the terminal point of $L(\sigma_1)$. Define $b_1$ and $b_2$ to be disjoint simple arcs on $\bd D$ with $b_1$ going from the initial point of $L(\sigma_1)$ to the terminal point of $L(\sigma_2)$ and $b_2$ going from the initial point of $L(\sigma_2)$ to the terminal point of $L(\sigma_1)$.

We can define three $(n-1)$-singular links from $L$ in the following manner:
	\begin{align}
		L_+ 	&= \overline{L(P\setminus\sigma_1)}\cup a_1;\notag\\
		L_- 	&= \overline{L(P\setminus\sigma_1)}\cup a_2;	\label{eqn:resl'ns}\\
		L_0 	&= \overline{L(P\setminus(\sigma_1\cup\sigma_2))}\cup(b_1\cup b_2).\notag
	\label{eqn:resl'ns}
	\end{align}
Note that a different choice of $a_1$ and $a_2$ would interchange $L_+$ and $L_-$. If $M$ is oriented, this ambiguity can be dealt with. Since $x$ is a transverse double point, consider the ordered pair of tangent vectors $L'(\sigma_1), L'(\sigma_2)$ at $x$. There is a unique vector normal to $D$ at $x$ that completes this ordered pair to an oriented frame that agrees with the orientation on $B$ at $x$. Such a vector points into one of the components of $B\setminus D$. Take $a_1$ to be in the same component.

This treatment of $n$-singular links in $M$ was utilized by Kalfagianni and Lin (\cite{KL}) to study homotopies of links in rational homology spheres. In particular, they saw that any free homotopy of a link in $M$ can be perturbed slightly to a particularly nice form called \emph{almost general position}, where the homotopy only fails to be isotopy at finitely many moments, and near these times one passes from some $L_+$ to $L_-$ or vice versa. Almost general position of homotopies was an important tool in \cite{KL} for the definition of the power series invariant that satisfies the HOMFLY skein relation. We briefly review their theorem. 

Let $U$ be the isotopy class of a knot in $M$ that is the standard unknot in a small ball neighborhood of a point of $M$. Fix $\scr{TL}$, a collection of links in $M$ such that there is exactly one representative in $\scr{TL}$ for each homotopy class of links in $M$, and furthermore, if $TL\in\scr{TL}$ has $k$ components that are homotopically trivial then $TL=L\coprod^k U$ for some link $L$ that has no homotopically trivial components. An element of $\scr{TL}$ is called a \emph{trivial link}. 

Let $\widehat R:=\bb C[[x,y]]$ be the ring of formal power series in the variables $x$ and $y$ over $\bb C$. Define $v\in\widehat R$ by $v:=e^y=1+y+\frac{y^2}2+\frac{y^3}6+\cdots$ and let $z\in\widehat R$ be defined by $z:=e^x-e^{-x}=2x+\frac{x^3}3+\frac{x^5}{60}+\cdots$.

E. Kalfagianni and X.S. Lin show the following in \cite{KL}. Recall that a 3-manifold $M$ is called \emph{atoroidal} if $\pi_2(M)$ is trivial and $M$ contains no essential tori. $M$ is called \emph{Seifert fibered} if it can be realized as an $S^1$-bundle over a 2-dimensional orbifold. 

\begin{thm}[\cite{KL}] 
Let $M$ be a rational homology 3-sphere that is either atoroidal or Seifert fibered over an orientable orbifold. Let $\scr L$ be the set of isotopy classes of links in $M$. Then, given values $J_M(T)$ for each $T\in\scr{TL}$ such that $J_M(T\coprod U)=\frac{v^{-1}-v}zJ_M(T)$, there is a unique map $J_M:\scr L\to\widehat R$ such that
	\[v^{-1}J_M(L_+)-vJ_M(L_-)=zJ_M(L_0).\]
	
\label{thm:psInvt}
\end{thm}

\subsection{Links and grid diagrams in $L(p,q)$}
\label{subsec:GridDiags}

In this section we review a construction central to this paper, that of toroidal grid diagrams in $L(p,q)$, developed in \cite{BG} (see also \cite{BGH}). Before doing so, we point out that, similar to the correspondence between Legendrian links and grid diagrams in the standard contact structure on $S^3$, there is a correspondence between Legendrian links in $L(p,q)$ with a universally tight contact structure and toroidal grid diagrams. This correspondence was fully developed in \cite{BG} and, along with the results of this paper, is used by the author in \cite{C2} to prove the Franks-Williams-Morton inequality in lens spaces with these contact structures.

The following definition of a toroidal grid diagram follows that given in \cite{BG}.

\begin{defn} A \emph{(toroidal) grid diagram $D$ with grid number $n$} in $L(p,q)$ is a set of data $(T,\vec\alpha, \vec\beta, \vec{\bb O},\vec{\bb X})$, where:
	\begin{itemize}
		\item $T$ is the oriented torus obtained via the quotient of $\rls^2$ by the $\ints^2$ lattice generated by $(1,0)$ and $(0,1)$.
		\item $\vec\alpha=\set{\alpha_0,\ldots,\alpha_{n-1}}$, with $\alpha_i$ the image of the line $y=\frac in$ in $T$. Call the $n$ annular components of $T-\vec\alpha$ the \emph{rows} of the grid diagram.
		\item $\vec\beta=\set{\beta_0,\ldots,\beta_{n-1}}$, with $\beta_i$ the image of the line $y=-\frac p{q}(x-\frac i{pn})$ in $T$. Call the $n$ annular components of $T-\vec\beta$ the \emph{columns} of the grid diagram.
		\item $\vec{\bb O}=\set{O_0, \ldots, O_{n-1}}$ is a set of $n$ points in $T-\vec\alpha-\vec\beta$ such that no two $O_i$'s lie in the same row or column.
		\item $\vec{\bb X}=\set{X_0, \ldots,X_{n-1}}$ is a set of $n$ points in $T-\vec\alpha-\vec\beta$ such that no two $X_i$'s lie in the same row or column.
	\end{itemize}
	
	The components of $T-\vec\alpha-\vec\beta$ are called the \emph{fundamental parallelograms} of $D$ and the points $\vec{\bb O}\cup\vec{\bb X}$ are called the \emph{markings} of $D$. Two grid diagrams with corresponding tori $T_1, T_2$ are considered equivalent if there exists an orientation-preserving diffeomorphism $T_1\to T_2$ respecting the markings (up to cyclic permutation of their labels).
\label{gridDefn}
\end{defn}

We often leave out the descriptor ``toroidal'' when referring to grid diagrams in $L(p,q)$. Note that such a grid diagram has ``slanted'' $\beta$ curves. For considerations of both convenience and aesthetics, we alter the fundamental domain of $T$ and ``straighten'' our pictures so that the $\beta$ curves are vertical. Figure \ref{fig:Strait} shows how this ``straightening'' is accomplished. Keep in mind that, as shown in Figure \ref{fig:Strait}, a fixed $\alpha_i$ is intersected $p$ times by each $\beta_j$, and $\beta_j$ is seen in the ``straightened'' figure as $p$ vertical arcs.

\begin{figure}[ht]
	\[\begin{tikzpicture}[>=stealth]
	\fill[green!60!black,fill opacity=0.70]
		(8.01,3.49) -- (8.51,1.76) -- (8.74,1.76) -- (8.24,3.49)--cycle;
	\fill[green!60!black, fill opacity=0.70]
		(6,-1.625) -- (6.375,-1.625) -- (6.375,-1.25) -- (6,-1.25) --cycle;
	\draw[green!30!black,->]
		(9,4.05) -- node [above, at start] {fundamental parallelogram} (8.35,2.8);
		
		\draw[step=3.5,thick]	(0,0) grid (3.5,3.5);
	\foreach \x in {2,3,4,5,6,7}
		\draw[thick] (0.5*\x,0)--(0.5*\x-1,3.5);
	\foreach \x in {2,3,4,5,6}
		\draw[gray,thin] (0.5*\x+0.25,0)--(0.5*\x-0.75,3.5);
	\foreach \r in {0,180}
		\draw[rotate around = {\r:(1.75,1.75)},thick]
			(0.5,0) -- (0,1.75);
	\foreach \x in {0,1}
	\foreach \r in {0,180}
		\draw[gray,thin,rotate around={\r:(1.75,1.75)}]
			(0.25+0.5*\x,0) -- (0,0.875+1.75*\x);
	\draw[gray,thin]
		(0,1.75) -- (3.5,1.75);
	\draw
		(0.375,0.875) node {x}
		(0.375,2.625) node {o}
		(2.125,2.625) node {x}
		(3.125,0.875) node {o};

	\draw[step=3.5,cm={1,0,0,1,(7,0)},thick]
		(0,0) grid (3.5,3.5);
	\draw[white,dashed,very thick]
		(10.5,0)--(10.5,3.5)--(9.5,3.5);
	\draw[thick]
		(6,3.5)--(7,3.5);
	\foreach \x in {0,1,2,3,4,5,6,7}
		\draw[cm={1,0,0,1,(7,0)},thick] (0.5*\x,0)--(0.5*\x-1,3.5);
	\foreach \x in {0,1,2,3,4,5,6}
		\draw[gray,thin,cm={1,0,0,1,(7,0)}] (0.5*\x+0.25,0)--(0.5*\x-0.75,3.5);
	\draw[gray,thin,cm={1,0,0,1,(7,0)}]
		(-0.5,1.75) -- (3,1.75);
	\draw[cm={1,0,0,1,(7,0)}]
		(0.375,0.875) node {x}
		(0.375,2.625) node {o}
		(2.125,2.625) node {x}
		(3.125,0.875) node {o};
	
	\draw[step=0.75,cm={1,0,0,1,(3,-2)},thick]
		(0,0) grid (5.25,0.75);
	\draw[gray,thin,step=0.75,cm={1,0,0,1,(3.375,-1.625)}]
		(-0.375,-0.375) grid (4.875,0.375);
	\draw[cm={1,0,0,1,(3,-2)}]
		(0.9375,0.1875) node {x}
		(1.6875,0.5625) node {o}
		(5.0625,0.1875) node {o}
		(4.3125,0.5625) node {x};
	\foreach \x in {0,7}
	\foreach \y in {0,3.5}
	\filldraw
		(\x,\y) circle (1pt);
	\foreach \x in {0,7}
	\draw
		(\x,0) node [below] {$z$}
		(\x,3.5) node [above] {$z$};
	\filldraw
		(3,-2) circle (1pt)
		(4.5,-1.25) circle (1pt);
	\draw
		(3,-2) node [below] {$z$}
		(4.5,-1.25) node [above] {$z$};
	\draw[thick]
		(-0.6,1.75) node {(A)}
		(6,1.75) node {(B)}
		(2.4,-1.625) node {(C)};
	
	\foreach \t in {0,3.5}
	\draw[thick,red]
		(3.5+\t,0) -- (2.5+\t,3.5) -- (3.5+\t,3.5)--cycle;
	\draw[thick, blue,->]
		(0.5,4.25) -- node [above,at start] {$\alpha_0$} (1.375,3.505);
	\draw[thick,blue,->]
		(-0.25,2.5) -- node [above,at start] {$\alpha_1$} (0.625,1.755);
	\draw[thick, blue,->]
		(0.8,-0.4) -- node [below,at start] {$\beta_0$} (0.85,0.5);
	\draw[thick,blue,->]
		(2.05,-0.4) -- node [below,at start] {$\beta_1$} (2.1,0.5);
		
	\end{tikzpicture}\]
\caption{(A) shows a grid diagram in $L(7,2)$, with grid number 2, on a fundamental domain of $T$. In (B) we alter the fundamental domain. (C) is the ``straightening'' of (B). As indicated, $\alpha_0$ is the thick horizontal circle in the figure; $\alpha_1$ is the thinner, grey horizontal circle; $\beta_0$ is the circle made of the thick slanted segments (which are vertical in (C)); and $\beta_1$ is the union of thin grey slanted segments.}
\label{fig:Strait}
\end{figure}
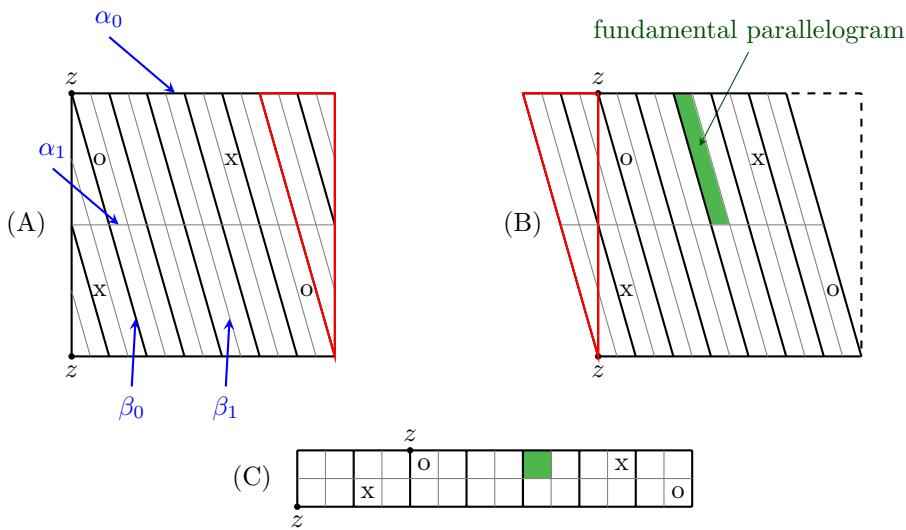

We may consider $L(p,q)$ as the quotient of $S^3\subset\bb C^2$ by the equivalence relation $(u_1,u_2)\sim(\omega_pu_1,\omega_p^qu_2)$, where $\omega_p=e^{\frac{2\pi i}p}$ (this giving the same manifold as $-p/q$ surgery on the unknot). Let $\pi:S^3\to L(p,q)$ be the quotient map. Represent points $(u_1,u_2)$ of $S^3$ in polar coordinates, setting $u_i=(r_i, \theta_i)$. The points of $L(p,q)$ can be identified with points in a fundamental domain of the cyclic action on $S^3$. Thus, since $r_2$ is determined by $r_1$ in $S^3$, we can describe $L(p,q)$ by
	\[L(p,q)=\setn{(r_1,\theta_1,\theta_2)}{r_1\in[0,1],\ \theta_1\in\left[0,2\pi\right),\ \theta_2\in\left[0,\frac{2\pi}p\right)}.\]

We identify the torus in the definition of a grid diagram with one that gives a Heegaard splitting of $L(p,q)$. Fix $\Sigma$ to be the torus in $L(p,q)$ defined by $r_1=\frac1{\sqrt 2}$. Then $\Sigma$ splits $L(p,q)$ into two solid tori $V^\alpha$ and $V^\beta$ 
\comment{
	\[V^\alpha=\setn{(r_1,\theta_1,\theta_2)}{r_1\in\left[0,\frac1{\sqrt2}\right]},\qquad\qquad V^\beta=\setn{(r_1,\theta_1,\theta_2)}{r_1\in\left[\frac1{\sqrt2},1\right]},\]
}
	with $\Sigma=\bd V^\alpha=-\bd V^\beta$.
	
An oriented link in $L(p,q)$ is associated to a grid diagram $D$ in $L(p,q)$ as follows. Identify $T$ with $-\Sigma\subset L(p,q)$ such that the $\alpha$-curves of $D$ are negatively-oriented meridians of $V^\alpha$ and the $\beta$-curves are meridians of $V^\beta$. Next connect each $X$ to the $O$ in its row by an ``horizontal'' oriented arc (from $X$ to $O$) that is embedded in $T$ and disjoint from $\vec\alpha$. Likewise, connect each $O$ to the $X$ in its column by a ``vertical'' oriented arc embedded in $T$ and disjoint from $\vec\beta$. The union of the $2n$ arcs makes a multicurve $\gamma$. Remove self-intersections of $\gamma$ by pushing the interiors of horizontal arcs up into $V^\alpha$ and the interiors of vertical arcs down into $V^\beta$. 

\begin{rem}
We note that our association of a link $K$ to a grid diagram $D$ in $L(p,q)$ differs slightly from that given in \cite{BG}, where the same grid diagram is associated to the mirror of $K$ in $L(p,q')$ with $qq'\equiv-1\mod p$. However, note that \cite{BG} provides a way to obtain our construction of $K$ from $D$. There is a related diagram called the \emph{dual grid diagram} and denoted $D^*$. If we take the link given by $D^*$ as constructed in \cite{BG} and associate the corresponding projection to $D$, this gives our construction.
\end{rem}
\begin{rem}
No part of Definition \ref{gridDefn} prohibits a marking in $\bb X$ and a marking in $\bb O$ from being in the same fundamental parallelogram. To a grid diagram that has grid number one (and so, only one marking in $\bb X$ and one marking in $\bb O$) and its two markings in the same fundamental parallelogram, we associate a knot in $L(p,q)$ that is contained in a small ball neighborhood and bounds an embedded disk. 
\label{rem:unknot}
\end{rem}

\begin{rem} Except for the case described in Remark \ref{rem:unknot}, we assume that each marking of $D$ is the center point of the fundamental parallelogram that contains it.  Let the straightened fundamental domain of $T$ have normalized coordinates $\setn{(\theta_1,\theta_2)}{\theta_1\in[0,p),\ \theta_2\in[0,1)},$ so that each $O$ and $X$ sharing the same column have the same $\theta_1$-coordinate (mod 1). 
\label{rpRem}
\end{rem}

\begin{defn}
Let $K$ be the (oriented) link associated to a grid diagram $D$ in $L(p,q)$ with grid number $n$. Then $K$ has a proper sublink $K'$ if and only if for some $0<m<n$ there are $m$ rows and $m$ columns of $D$ so that the $2m$ markings in these $m$ rows are exactly the $2m$ markings in the $m$ columns. By identifying the boundaries of these $m$ rows and $m$ columns so that their relative order is preserved, we get a grid diagram $D'$ with grid number $m$ with associated link $K'$. If this $K'$ has one component then $D'$ is called a \emph{component of $D$}.
\label{defn:compDiag}
\end{defn}

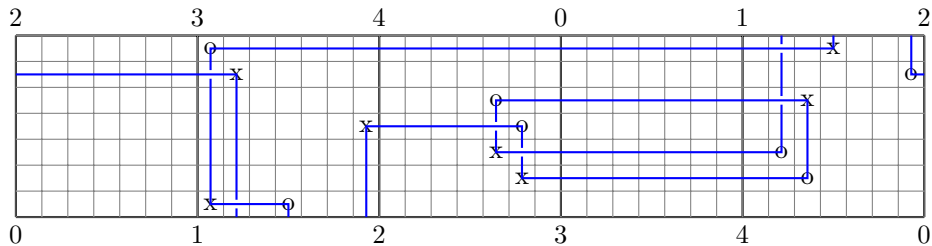
\begin{figure}[ht]
	\[\begin{tikzpicture}[scale=0.115]
		\draw[step=21,thick]
			(0,0) grid (105,21);
		\draw[gray,thin,step=3]
			(0,0) grid (105,21);	
	\foreach \o in {(22.5,19.5),(31.5,1.5),(55.5,13.5),(58.5,10.5),(88.5,7.5),(91.5,4.5),(103.5,16.5)}
		\draw
			\o node {o};
	\foreach \x in {(25.5,16.5),(22.5,1.5),(40.5,10.5),(55.5,7.5),(58.5,4.5),(91.5,13.5),(94.5,19.5)}
		\draw
			\x node {x};
	\draw[blue,thick]
		(0,16.5) -- (25.5,16.5) -- (25.5,2)
		(25.5,1) -- (25.5,0)
		(31.5,0) -- (31.5,1.5) -- (22.5,1.5) -- (22.5,16)
		(22.5,17) -- (22.5,19.5) -- (94.5,19.5) -- (94.5,21)
		(40.5,0) -- (40.5,10.5) -- (58.5,10.5) -- (58.5,8)
		(58.5,7) -- (58.5,4.5) -- (91.5,4.5) -- (91.5,13.5) -- (55.5,13.5) -- (55.5,11)
		(55.5,10) -- (55.5,7.5) -- (88.5,7.5) -- (88.5,13)
		(88.5,14) -- (88.5,19)
		(88.5,20) -- (88.5,21)
		(103.5,21) -- (103.5,16.5) -- (105,16.5);
	\foreach \l in {0,1,2,3,4}
	\draw 
		(21*\l,0) node [below] {\l};
	\draw	(105,0) node [below] {0};	
	\foreach \l in {0,1,2}
	\draw
		(21*\l+63,21) node[above] {\l};
	\foreach \l in {2,3,4}
	\draw
		(21*\l-42,21) node[above] {\l};
	
	\end{tikzpicture}\]
	\caption{A grid diagram for a link in $L(5,3)$ with corresponding grid projection.}
	\label{gdexamFig}
\end{figure}

Under the requirements of Remark \ref{rpRem}, the projection of $K$ to $-\Sigma$ (with vertical arcs crossing under horizontal arcs) is called a \emph{grid projection} associated to $D_K$ (the authors of \cite{BG} call this a rectilinear projection). Note that $K$ has an orientation given by construction and so the grid projection is also oriented. Figure \ref{gdexamFig} shows an example of a grid diagram with a corresponding grid projection.

If $D_K$ has grid number $n$ then there are $2^{2n}$ different grid projections associated to $D_K$, as there are two choices for each vertical and each horizontal arc. These choices, however, do not affect the isotopy type of the associated link. (In fact, \cite[Proposition 4.6]{BG} shows that the Legendrian isotopy class of the link is independent of the choice of grid projection.) The difference in choice of arc in a particular row or column corresponds to isotoping the link across a meridional disk of $V_\alpha$ or $V_\beta$ respectively. We will refer to such a move on a grid projection as a \emph{disk slide}.

The disk slide is a move that is undetected by grid diagrams: it changes the particular grid projection, but the underlying grid diagram remains unchanged. However, there is a set of moves on a grid diagram in $L(p,q)$ that correspond to an isotopy of the associated link and which are sufficient to interpolate between any two grid diagrams for a given link $K$. As these grid moves are used in what follows, we describe them here. The moves come in two flavors: grid (de)stabilizations and commutations. 

{\bf Grid Stabilizations and Destabilizations:} Grid stabilizations increase the grid number by one and should be thought of as adding a local kink to the knot. They are named with an X or O, depending on the type of marking at which stabilization occurs, and with NW,NE,SW, or SE, depending on the positioning of the new markings. Figure \ref{fig:stabl} shows how an X:NW stabilization and an O:SW stabilization would affect a diagram with grid number 1. Destabilizations are the inverse of a stabilization. Any (de)stabilization is a grid move that preserves the isotopy type.

\begin{figure}[ht]
	\begin{tikzpicture}[>=stealth]
	\draw[step=0.75,cm={1,0,0,1,(0,0)},thick]
		(0,0) grid (5.25,0.75);
	\draw
		(1.125,0.375) node {X}
		(4.875,0.375) node {O};
	\draw[->]
		(5.5,0.5) -- (6.5,0.5);
	\draw[step=0.75,cm={1,0,0,1,(6.7,0)},thick]
		(0,0) grid (5.25,0.75);
	\draw[gray,thin,step=0.75,cm={1,0,0,1,(7.075,0.375)}]
		(-0.375,-0.375) grid (4.875,0.375);
	\draw[cm={1,0,0,1,(6.7,0)}]
		(0.9375,0.1875) node {x}
		(1.3125,0.1875) node {o}
		(4.6875,0.5625) node {o}
		(1.3125,0.5625) node {x};
	\draw[step=0.75,cm={1,0,0,1,(0,-2)},thick]
		(0,0) grid (5.25,0.75);
	\draw[cm={1,0,0,1,(0,-2)}]
		(1.125,0.375) node {X}
		(4.875,0.375) node {O};
	\draw[->,cm={1,0,0,1,(0,-2)}]
		(5.5,0.5) -- (6.5,0.5);
	\draw[step=0.75,cm={1,0,0,1,(6.7,-2)},thick]
		(0,0) grid (5.25,0.75);
	\draw[gray,thin,step=0.75,cm={1,0,0,1,(7.075,-1.625)}]
		(-0.375,-0.375) grid (4.875,0.375);
	\draw[cm={1,0,0,1,(6.7,-2)}]
		(0.9375,0.1875) node {x}
		(1.3125+3.75,0.1875) node {o}
		(4.6875,0.5625) node {o}
		(1.3125+3.75,0.5625) node {x};

	\end{tikzpicture}
\caption{Stabilization of type X:NW (top) and of type O:SW (bottom)}
\label{fig:stabl}
\end{figure}
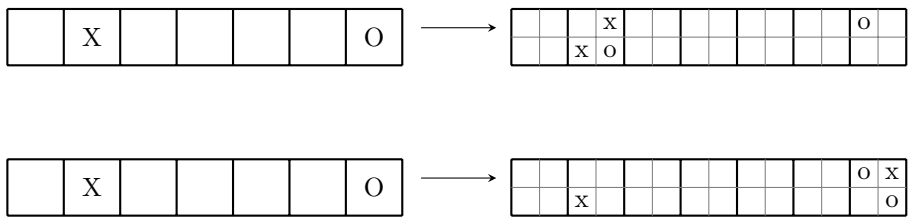

{\bf Commutations:} A commutation interchanges two adjacent columns (or rows) of the grid diagram (except as in Figure \ref{fig:illegalcommut}). Let $A$ be the annulus consisting of the two adjacent columns $c_1,c_2$ (resp.\ rows $r_1,r_2$) involved in the commutation. This annulus is sectioned into $pn$ segments of the $n$ rows (resp.\ columns) of the grid diagram. Let $s_1, s_1'$ be the two segments in $A$ containing the markings of $c_1$ (resp.\ $r_1$). If the markings of $c_2$ (resp.\ $r_2$) are contained in separate components of $A-s_1-s_1'$, the commutation is called \emph{interleaving}. If they are in the same component of $A-s_1-s_1'$ the commutation is called \emph{non-interleaving}. We note that in the literature a commutation typically refers only to what we call a non-interleaving commutation. We have extended the terminology to include the interleaving case. An example of non-interleaving commutation is shown in Figure \ref{fig:commut}.

\begin{figure}[ht]
	\[\begin{tikzpicture}[scale=0.15,>=stealth]
		\draw[step=8,thick]
			(0,0) grid (56,8);
		\draw[step=2,gray!50!white,thin]
			(0,0) grid (56,8);
	
	\draw[blue, thick]
		(5,5) -- (5,0)
		(21,8) -- (21,5.5) (21,4.5) -- (21,1) -- (7,1) -- (7,0)
		(23,8) -- (23,5.5) (23,4.5) -- (23,0)
		(39,8) -- (39,7) -- (56,7)
		(0,7) -- (3,7) -- (3,8)
		(43,0) -- (43,2.5) (43,3.5) -- (43,6.5) (43,7.5) -- (43,8)
		(27,0) -- (27,4.5) (27,5.5) -- (27,8)
		(11,0) -- (11,0.5) (11,1.5) -- (11,4.5) (11,5.5) -- (11,8)
		(51,0) -- (51,6.5) (51,7.5) -- (51,8)
		(35,0) -- (35,3) -- (49,3) -- (49,6.5) (49,7.5) -- (49,8)
		(33,0) -- (33,5) -- (5,5);
	
	\draw
		(3,7) node {x}
		(35,3) node {o}
		(5,5) node {x}
		(21,1) node {o};
	\draw
		(49,3) node {x}
		(33,5) node {o}
		(7,1) node {x}
		(39,7) node {o};
		
	\foreach \l in {0,1,2,3,4,5,6}
	\draw 
		(8*\l,0) node [below] {\l};
	\draw	(56,0) node [below] {0};	
	\foreach \l in {0,1,2,3,4,5}
	\draw
		(8*\l+16,8) node[above] {\l};
	\foreach \l in {5,6}
	\draw
		(8*\l-40,8) node[above] {\l};
	
	\foreach \c in {0,1,2,3,4,5,6}
		\draw[<->]
			(3+8*\c,9) -- (5+8*\c,9);
	\draw[step=8,thick]
		(24,-16) grid (80,-8);
	\draw[step=2,gray!50!white,thin]
		(24,-16) grid (80,-8);	

	\draw[thick,blue,cm={1,0,0,1,(24,-16)}]
		(3,5) -- (3,0)
		(19,8) -- (19,5.5) (19,4.5) -- (19,1) -- (7,1) -- (7,0)
		(23,8) -- (23,5.5) (23,4.5) -- (23,0)
		(39,8) -- (39,7) -- (56,7)
		(0,7) -- (5,7) -- (5,8)
		(45,0) -- (45,2.5) (45,3.5) -- (45,6.5) (45,7.5) -- (45,8)
		(29,0) -- (29,4.5) (29,5.5) -- (29,8)
		(13,0) -- (13,0.5) (13,1.5) -- (13,4.5) (13,5.5) -- (13,8)
		(53,0) -- (53,6.5) (53,7.5) -- (53,8)
		(37,0) -- (37,3) -- (49,3) -- (49,6.5) (49,7.5) -- (49,8)
		(33,0) -- (33,5) -- (3,5);
	\draw[cm={1,0,0,1,(24,-16)}]
		(5,7) node {x}
		(37,3) node {o}
		(3,5) node {x}
		(19,1) node {o};
	\draw[cm={1,0,0,1,(24,-16)}]
		(49,3) node {x}
		(33,5) node {o}
		(7,1) node {x}
		(39,7) node {o};
		
	\foreach \l in {0,1,2,3,4,5,6}
	\draw [cm={1,0,0,1,(24,-16)}]
		(8*\l,0) node [below] {\l};
	\draw[cm={1,0,0,1,(24,-16)}]	(56,0) node [below] {0};	
	\foreach \l in {0,1,2,3,4,5}
	\draw[cm={1,0,0,1,(24,-16)}]
		(8*\l+16,8) node[above] {\l};
	\foreach \l in {5,6}
	\draw[cm={1,0,0,1,(24,-16)}]
		(8*\l-40,8) node[above] {\l};
	\draw[->,thick]
		(8,-5) .. controls (4,-12) and (21,-12) .. (23,-12);
	
	\end{tikzpicture}\]
	\caption{A non-interleaving commutation in $L(7,2)$}
	\label{fig:commut}
\end{figure}
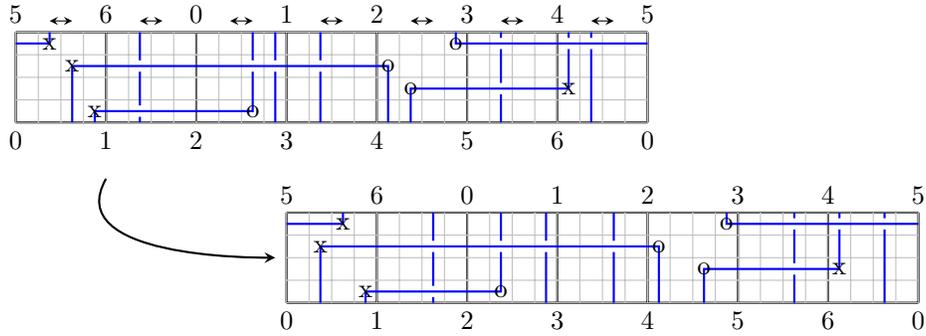

\begin{rem}
We emphasize that the definition of a commutation does not include the case depicted in Figure \ref{fig:illegalcommut}, where one row (resp.\ column) contains markings of both $c_1$ and $c_2$ (resp.\ $r_1,r_2$), for then this marking of $c_2$ is not in either component of $A-s_1-s_1'$. While such an exchange of columns or rows would not change the knot type, it does not correspond to a Legendrian isotopy. Though we do not deal in this paper with the correspondence between grid diagrams and Legendrian links, the techniques we develop here are intended to be useful for studying such. This motivates our definition of a commutation, which, if non-interleaving, is a grid move that preserves Legendrian isotopy type \cite{BG}.
\label{rem:illegalcommut}
\end{rem}

\begin{figure}[ht]
	\begin{tikzpicture}[scale=0.055,>=stealth]
		\foreach \t in {0,70}
		\draw[gray!50!white,thin]
			(0+\t,0) -- (0+\t,40)
			(10+\t,0) --(10+\t,40)
			(20+\t,0) --(20+\t,40);
		\draw[blue,thick]
			(5,0) -- (5,20) -- (15,20) -- (15,40);
		\draw[blue,thick]
			(85,0) --(85,20)-- (75,20) -- (75,40);
		\draw
			(5,20) node {x}
			(15,20)node {o};
		\draw
			(85,20)node {x}
			(75,20)node {o};
	\draw[<->] (5,45) -- (15,45);
	\draw[gray,->,line width=10pt] (63-25,20) -- (82-25,20);
	\draw[line width=6pt,white] (62-25,20) -- (73-25,20);

	\end{tikzpicture}
	\caption{A move which is neither an interleaving nor non-interleaving commutation}
	\label{fig:illegalcommut}
\end{figure}
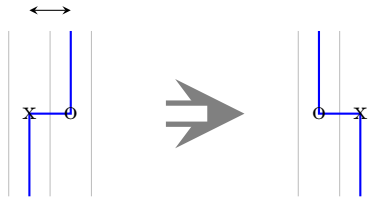

We end the section with a construction that plays an important role in what follows. Given a grid diagram $D$ in $L(p,q)$ define a grid diagram $\wt{D}$ of a link in $S^3$ as follows: cut $T$ along $\alpha_0$ to get an annulus $A$. The boundary of $A$ is a disjoint union of two copies of $\alpha_0$. Let one of these copies be $\alpha_0^+$ and the other be $\alpha_0^-$. Take $p$ copies of $A$, say $A_0,\ldots,A_{p-1}$ and glue $\alpha_0^+$ on $A_i$ to $\alpha_0^-$ on $A_{i+1 (\text{mod }p)}$ for $i=0,\ldots,p-1$ by the identity map. Note that the torus constructed from $A_0,\ldots,A_{p-1}$ covers $T$ under the cover $\pi:S^3\to L(p,q)$ and so the link associated to $\wt D$ covers the link associated to $D$. Call $\wt D$ the \emph{lift of $D$ to $S^3$}. (An example is shown in Figure \ref{fig2:example1}, where the grid diagram on the right is the lift of the grid diagram in $L(5,1)$ on the left. The link corresponding to the lift is a Hopf link.)

\begin{figure}[ht]
	\[\begin{tikzpicture}[scale=0.08]
	\draw[cm={1,0,0,1,(-70,0)}]
		(25,-10) node {{\large $D$}};
	\draw[step=10,thick,cm={1,0,0,1,(-70,20)}]
			(0,0) grid (50,10);
	\draw[gray,thin,step=5,cm={1,0,0,1,(-70,20)}]
			(0,0) grid (50,10);	
	\foreach \l in {0,1,2,3,4}
	\draw[cm={1,0,0,1,(-70,20)}] 
		(10*\l,0) node [below] {\l};
	\draw[cm={1,0,0,1,(-70,20)}]	
		(50,0) node [below] {0};	
	\foreach \l in {0,1,2,3,4}
	\draw[cm={1,0,0,1,(-70,20)}]
		(10*\l+10,10) node[above] {\l};
	\draw[cm={1,0,0,1,(-70,20)}]
		(0,10) node[above] {4};
	
	
	\foreach \x in {(42.5,7.5), (17.5,2.5)}
		\draw[cm={1,0,0,1,(-70,20)}]
		 	\x node {x};
	\foreach \o in {(2.5,7.5), (7.5,2.5)}
		\draw[cm={1,0,0,1,(-70,20)}]
			\o node {o};

	\draw
		(25,-10) node {{\large $\wt D$}};	
	\draw[step=10,thick]
			(0,0) grid (50,50);
	\draw[gray,thin,step=5]
			(0,0) grid (50,50);	
	\foreach \l in {0,1,2,3,4}
	\draw 
		(10*\l,0) node [below] {\l}
		(10*\l,50) node [above] {\l};
	\draw	
		(50,0) node [below] {0}
		(50,50) node [above] {0};	
	\foreach \c in {0,1,2,3}
	\foreach \x in {(17.5+10*\c,2.5+10*\c)}
		\draw
		 	\x node {x};
		\draw
			(7.5,42.5) node {x};
		\draw
			(42.5,7.5) node {x};
	\foreach \c in {1,2,3,4}
	\foreach \x in {(42.5-50+10*\c,7.5+10*\c)}
		\draw
		 	\x node {x};
	\foreach \c in {0,1,2,3,4}
	\foreach \o in {(2.5+10*\c,7.5+10*\c), (7.5+10*\c,2.5+10*\c)}
		\draw
			\o node {o};

	\end{tikzpicture}\]
	\caption{$\wt D$: the lift to $S^3$ of the grid diagram $D$ in $L(5,1)$.}
	\label{fig2:example1}
\end{figure}
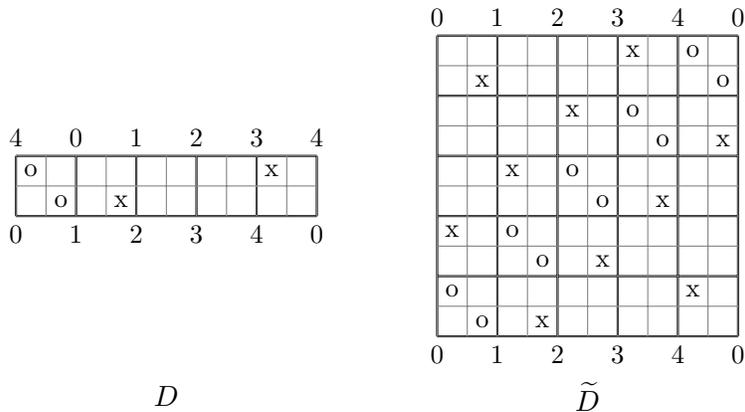

\section{Homotopy of links}
\label{sec:hty}

\begin{defn} For an oriented link $K$ in $L(p,q)$, define $\mu(K)$ to be the homology class of $K$ in $H_1(L(p,q))$.
\label{defn:mu}
\end{defn}

\begin{rem} Let $K, K'$ be oriented knots. We note that $\mu(K)=\mu(K')$ if and only if $K$ and $K'$ are freely homotopic, since the free homotopy class of an oriented knot in $L(p,q)$ is determined by its homology class.
\label{rem:mu}
\end{rem}

Let $C$ be the core of the handlebody $V_\alpha$ in $L(p,q)$ (dual to a meridional disk of $V_\alpha$). Then $H_1(L(p,q))\cong\ints/p$ is cyclically generated by $\gamma=[C]\in H_1(L(p,q))$. We identify $\mu(K)$ with the multiple of $\gamma$ it represents.

Suppose $D_K=(T,\vec\alpha,\vec\beta,\bb O,\bb X)$ is a grid diagram with associated link $K$ in $L(p,q)$. Orient the $\alpha$ and $\beta$ curves so that the algebraic intersection $\alpha_i\cdot\beta_j$ is positive on $T$, for all pairs $i,j$. In all subsequent figures, this means that each $\alpha_i$ is oriented to the right and each $\beta_j$ is oriented upwards. Let $R_K$ be a grid projection for $D_K$ with its given orientation. We can compute $\mu(K)$ from the grid diagram as follows.

\begin{lem}
Given $D_K$ and $R_K$ as above, choose an $\alpha$-curve $\alpha_i$ in the grid diagram. Then $\mu(K)$ is equal (mod $p$) to $\alpha_i\cdot R_K$, the algebraic intersection of $\alpha_i$ with $R_K$.
\label{lem:monoRP}
\end{lem}

\begin{proof}
Push the interiors of horizontal arcs on $R_K$ slightly into $V_\alpha$ to get a knot $K'$ contained in $V_\alpha$. The winding number of $K'$ in $V_\alpha$ is clearly counted by $\alpha_i\cdot R_K$ and $K'$ is isotopic to $K$, so $\mu(K')=\mu(K)$.
\end{proof}


\begin{figure}[ht]
	\[\begin{tikzpicture}[scale=0.3,thick,>=stealth]
		\draw [cm={1,0,0,1,(0,3)}]
			(5,10) -- (5,14.5) (5,15.5) -- (5,20)
			(0,15) -- (10,15);
		\draw[cm={1,0,0,1,(0,3)}]
			(20,16) -- (25,16) -- (25,15.5) (25,14.5) -- (25,14) -- (30,14)
			(26,20) -- (26,15) -- (24,15) -- (24,10);
		
		\foreach \c in {5,25}
		\draw
			(\c,15+3) node[circle,draw,dashed,minimum size=40,label=135:{$W$}]{};
			
		\foreach \x in {4,7}
		\draw[gray,thin] 
			(0,\x) -- (11,\x)
			(\x,0) -- (\x, 11);
		\draw
			(0,5.5) node {X}
			(5.5,0) node {X}
			(11,5.5) node {O}
			(5.5,11) node {O};

		\foreach \x in {4,5,6,7}
		\draw[gray,thin] 
			(20,\x) -- (31,\x)
			(\x+20,0) -- (\x+20, 11);
		
		\foreach \r in {0,90}
		\draw[rotate around={\r:(25.5,5.5)}]
			(20,6.5) node {x}
			(25.5,6.5) node {o}
			(25.5,4.5) node {x}
			(31,4.5) node {o};	
		
\foreach \tx in {17,5}
	\draw[gray,->,line width=9pt] (13.5,\tx) -- (16.7,\tx);
\foreach \tx in {17,5}
	\draw[line width=5pt,white] (13.4,\tx) -- (14.7,\tx);

	\end{tikzpicture}\]
\caption{A crossing change on a grid projection}
\label{fig:rectCrossChange}
\end{figure}
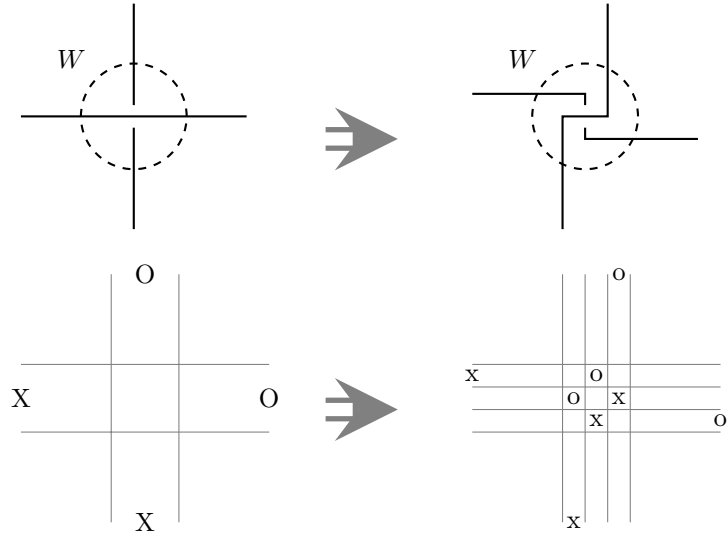

\subsection{Skein Theory}
\label{sec:skTheory}
We now develop a skein theory of grid diagrams in $L(p,q)$. Note that grid diagrams have been discussed in the setting of singular links in $S^3$, and used to extend link Floer homology to singular links \cite{Aud}. What we describe here as grid diagrams that differ by a ``skein crossing change'' are exactly the grid diagrams in \cite{Aud} that correspond to the two resolutions of the double point of a 1-singular link (see, for example, Figure 10 in \cite{Aud}).

If two links $L_+$ and $L_-$ differ as described in (\ref{eqn:resl'ns}) they are isotopic to links with associated grid diagrams differing by the grid move in Figure \ref{fig:rectCrossChange}. The next lemma refines this. In a figure of a grid diagram, slanted gray bars indicate some number of columns (and their markings) which may be between the columns of interest.

\begin{lem}
Let $L, L'$ be links corresponding to grid diagrams that differ by an interleaving commutation. Then $L, L'$ differ by a crossing change.
\label{trivlinksLem}
\end{lem}

\begin{figure}[ht]
	\begin{tikzpicture}[scale=0.8,>=stealth]
\foreach \tx in {0,8}
\foreach \ty in {0,-4,-8,-12}
	\foreach \c in {0,...,3}
	\draw[step=1,gray,thin,cm={1,0,0,1,(1.8*\c+\tx,\ty)}]
		(-0.2,0) grid (1.2,2.25);
	\foreach \c in {1/3,2/3}
	\foreach \ty in {0,-4}
	\foreach \tx in {0,-8}
	\draw[gray,thin]
		(11.6+\c+\tx,0+\ty+0.5*\tx) -- (11.6+\c+\tx,2.25+\ty+0.5*\tx)
		(7.8+\tx,\c+\ty+0.5*\tx) -- (14.6+\tx,\c+\ty+0.5*\tx);
	\foreach \ty in {0,-4}
	\draw[gray,thin,cm={1,0,0,1,(2*\ty,-8+\ty)}]
		(12.1-1.8/4*\ty,0) -- (12.1-1.8/4*\ty,2.25)
		(7.8,0.5) -- (14.6,0.5);
\foreach \tty in {0,-4,-8,-12}
\foreach \tx in {0,8}
\foreach \ty in {0,0.4,0.8,1.2,1.6}
	\foreach \b in {1.2,3,4.8}
	\fill[gray,cm={1,0,0,1,(\tx,\ty+\tty)}]
		(\b,0) -- (\b,0.2) -- (\b+0.4,0.5) -- (\b+0.4,0.3) -- cycle;
\foreach \tty in {0,-4,-8,-12}
\foreach \tx in {0,8}
\foreach \ty in {0.2,0.6,1,1.4}
	\foreach \b in {1.2,3,4.8}
	\fill[white,cm={1,0,0,1,(\tx,\ty+\tty)}]
		(\b,0) -- (\b,0.2) -- (\b+0.4,0.5) -- (\b+0.4,0.3) -- cycle;
\foreach \tty in {0,-4,-8,-12}
\foreach \tx in {0,8}
	\foreach \b in {1.2,3,4.8}
	\fill[gray,cm={1,0,0,1,(\tx,\tty)}]
		(\b,2.25) -- (\b+0.33,2.25) -- (\b,2) -- cycle;
\foreach \tty in {0,-4,-8,-12}
\foreach \tx in {0,8}
	\foreach \b in {1.2,3,4.8}
	\fill[gray,cm={1,0,0,1,(\tx,\tty)}]
		(\b+0.4,0) -- (\b+0.4,0.1) -- (\b+0.277,0) -- cycle;
	\draw
		(0.5,1.5) node {X}
		(4.1,1.5) node {O}
		(2.3,0.5) node {X}
		(5.9,0.5) node {O};
	\draw
		(8.5,1.5) node {X}
		(10.3,5/6) node {x}
		(12.1,1/6) node {x}
		(12.1+1/3,0.5) node {x}
		(12.1-1/3,0.5) node {o}
		(12.1,5/6) node {o}
		(12.1+1/3,1.5) node {o}
		(13.9,1/6) node {o};
	\draw[cm={1,0,0,1,(-8,-4)}]
		(8.5,5/6) node {x}
		(10.3,1.5) node {X}
		(12.1,1/6) node {x}
		(12.1+1/3,0.5) node {x}
		(12.1-1/3,0.5) node {o}
		(12.1,1.5) node {o}
		(12.1+1/3,5/6) node {o}
		(13.9,1/6) node {o};
	\draw[cm={1,0,0,1,(0,-4)}]
		(8.5,5/6) node {x}
		(10.3,1.5) node {X}
		(12.1+1/3,1/6) node {x}
		(12.1,0.5) node {x}
		(12.1-1/3,0.5) node {o}
		(12.1+1/3,1.5) node {o}
		(12.1,5/6) node {o}
		(13.9,1/6) node {o};
	\draw[cm={1,0,0,1,(-8,-8)}]
		(8.5,0.5) node {x}
		(10.3,1.5) node {X}
		(12.1+1/3,5/6) node {x}
		(12.1,1/6) node {x}
		(12.1-1/3,1/6) node {o}
		(12.1+1/3,1.5) node {o}
		(12.1,0.5) node {o}
		(13.9,5/6) node {o};
	\draw[cm={1,0,0,1,(0,-8)}]
		(8.5,0.25) node {X}
		(10.3,1.5) node {X}
		(12.1+0.25,0.75) node {X}
		(12.1-0.25,0.25) node {O}
		(12.1+0.25,1.5) node {O}
		(13.9,0.75) node {O};
	\draw[cm={1,0,0,1,(-8,-12)}]
		(8.5,0.25) node {X}
		(10.3,1.5) node {X}
		(13.9-0.25,0.75) node {X}
		(12.1,0.25) node {O}
		(13.9-0.25,1.5) node {O}
		(13.9+0.25,0.75) node {O};
	\draw[cm={1,0,0,1,(0,-12)}]
		(8.5,0.5) node {X}
		(10.3,1.5) node {X}
		(12.1,0.5) node {O}
		(13.9,1.5) node {O};

	\draw[thick,red,<->]
		(14.7,1.5) --(14.7,5/6);
	\draw[thick,red,<->]
		(4,2.4-4)--(4.5,2.4-4);
	\draw[thick,red,<-]
		(14.7,-2.5-2/3)--(14.7,1/6-4);
	\foreach \a in {0,0.38}
	\draw[thick,red,<-, rotate around={90:(7.7,0.05+\a)},cm={1,0,0,1,(-0.3,1.65+\a)}]
		(4.05,2.4-4)--(4.4,2.4-4);
	\draw[very thick,red]
		(4.1-1/6,1/3-8) node [circle,draw,minimum size=20]{};
	\foreach \a in {0,0.75}
	\draw[very thick,red,->]
		(12.55,-6.5-\a) -- (13.3,-6.5-\a);
	\draw[very thick,red]
		(5.9,1.1-12) node [circle,draw,minimum size=30]{};
	\draw[red]
		(5,-9) node {\text{destabilize}};
	\draw[->, red]
		(5,-8.7) -- (4.13,-8.2);
	\draw[->, red]
		(5.4,-9.2) -- (5.75,-10.1);
	\foreach \l in {1,3,5,7}
	\draw
		(0,2.7-2*\l+2) node [circle] {\begin{Large}{\bf \l}\end{Large}};
	\foreach \l in {2,4,6,8}
	\draw
		(8,2.7-2*\l+4) node [circle] {\begin{Large}{\bf \l}\end{Large}};
	\end{tikzpicture}
\caption{${\bf1\to2}$:crossing change; ${\bf2\to3}$:row commutation; ${\bf3\to4}$:column commutation; ${\bf4\to5}$:two row commutations; ${\bf5\to6}$:destabilization; ${\bf6\to7}$:many column commutations; ${\bf7\to8}$:destabilization}
\label{interlvcaseFig}
\end{figure}
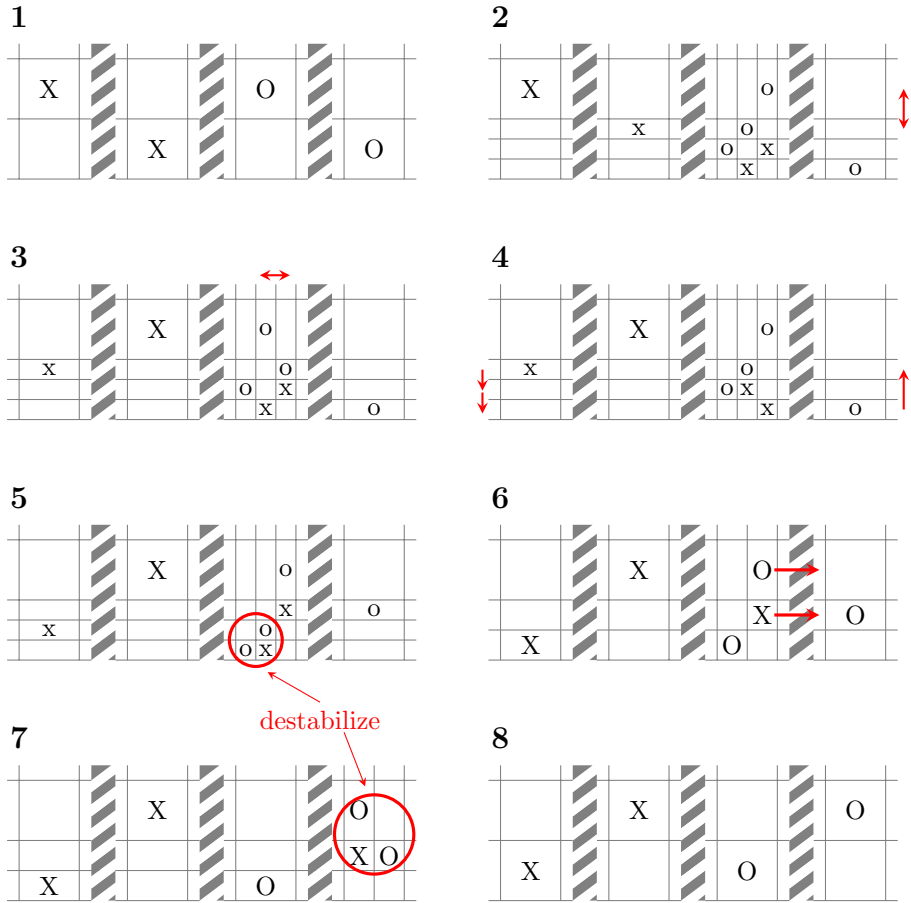

\begin{proof}
We prove the statement for row commutation. The case with columns is similar.

The proof is given by the sequence of grid moves detailed in Figure \ref{interlvcaseFig}, where arrows indicate commutations to be performed, and a triple of markings that is to be destabilized is circled. Referring to Figure \ref{interlvcaseFig}, there are 8 steps. From step 1 to 2 we perform a crossing change. From 2 to 5 we perform a number of non-interleaving commutations. Going from 5 to 6 is destabilization. From 6 to 7 involves several commutations. Each of these commutations is non-interleaving since the markings of the column being moved to the right are in adjacent rows. Finally we destabilize from 7 to 8.

Note that all grid moves above correspond to isotopy of the link, except the first which, as in Figure \ref{fig:rectCrossChange}, corresponds to interchanging the two resolutions of a singular link defined by (\ref{eqn:resl'ns}).
\end{proof}

Lemma \ref{trivlinksLem} shows that if two grid diagrams differ only by the interleaving commutation of two adjacent columns (or rows), then the corresponding links are some pair $L_+, L_-$. Motivated by this result, we make the following definition.

\begin{defn}
A pair of adjacent columns in a grid diagram for $L(p,q)$ is called a \emph{skein crossing} if their commutation is an interleaving commutation. Define a \emph{skein crossing change} to be the commutation of a skein crossing pair. A skein crossing is called \emph{positive} if it appears as in the left side of Figure \ref{fig:signSkCross} under some cyclic permutation of the rows and columns. It is \emph{negative} if (after some cyclic permutation) it appears as in the right side of Figure \ref{fig:signSkCross}. Given two grid diagrams that differ only by a skein crossing change, we call the diagram with the positive skein crossing $D_+$ and the diagram with the negative skein crossing $D_-$.
\label{defn:skCross}
\end{defn} 

\begin{figure}[ht]
	\[\begin{tikzpicture}
\foreach \t in {2,6}
	\draw[gray,thin]
		(\t-0.25,-0.1) -- (\t-0.25,1.6)
		(\t+0.25,-0.1) -- (\t+0.25,1.6)
		(\t+0.75,-0.1)-- (\t+0.75,1.6);
	\draw
		(2,0) node {o}
		(2,1) node {x}
		(2.5,0.5) node {o}
		(2.5,1.5) node {x};
	\draw
		(2.25,-0.5) node {$D_+$};
	\draw
		(6,1.5) node {x}
		(6,0.5) node {o}
		(6.5,0) node {o}
		(6.5,1) node {x};
	\draw
		(6.25,-0.5) node {$D_-$};
	\end{tikzpicture}\]
\caption{Positive and negative skein crossings.}
\label{fig:signSkCross}
\end{figure}
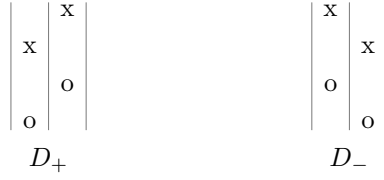

There is a function on grid diagrams that is related to the idea of skein crossings. We will use this function later, both to define trivial links and to produce a complexity function on grid diagrams that will prove Theorem \ref{mainThm}.

\begin{defn}
Consider a (not necessarily adjacent) pair of columns $c_1, c_2$ in a grid diagram of $L(1,0)=S^3$. Let $R_i(\text o)$ be the row in the diagram that contains the $\bb O$ marking of $c_i$ and let $R_i(\text x)$ be the row containing the $\bb X$ marking of $c_i$. Call the columns $c_1,c_2$ \emph{interleaving} if $R_2(\text o)$ and $R_2(\text x)$ are in different annular components of $T-(R_1(\text o)\cup R_1(\text x))$. Note that if $c_1,c_2$ are adjacent and interleaving then they comprise a skein crossing.

Given a grid diagram $D$ for a link in $L(p,q)$, define $scr(D)$ to be the number of interleaving pairs of columns in $\wt D$, the lift of $D$ to $S^3$. 
\label{defn:scr}
\end{defn}

Note that, despite the name of the function, $scr(D)$ does not count the number of skein crossings of $\wt D$. Instead, it counts all pairs of columns of $\wt D$ which, if they were adjacent, would make a skein crossing. Let us point out a few important properties of the function $scr$.

\begin{prop} Let $\scr{D}(L(p,q))$ denote the set of all grid diagrams in $L(p,q)$. The function $scr:\scr D(L(p,q))\to\ints$ satisfies the following:
	\begin{enumerate}
		\item (Orientation invariance) Given a grid diagram $D$, let $rD$ denote the grid diagram obtained by exchanging every $\bb O$ marking for an $\bb X$ marking and vice versa. Then $scr(D)=scr(rD)$.
		\item (Column commutation invariance) Let $D, D'$ be two grid diagrams that only differ by a column commutation. Then $scr(D)=scr(D')$. Note that this is not true for row commutation. 
	\end{enumerate}
\label{prop:scrProperties}
\end{prop}

\begin{proof} The fact that orientation invariance holds is just the observation that the property of a pair of columns of $\wt D$ being interleaving depends only on the relative positions of their markings, independent of which is in $\bb O$ and which is in $\bb X$.

To prove the second property, first note that an adjacent pair of columns in $D$ will correspond to $p$ adjacent pairs of columns in $\wt D$. Thus, if $D$ and $D'$ differ by a column commutation, then $\wt D$ and $\wt{D'}$ differ by $p$ column commutations. Given a column $c$ of $\wt D$, if $R(\text o)$ and $R(\text x)$ are the rows that its markings occupy, then its markings are still in rows $R(\text o)$ and $R(\text x)$ after column commutation. It is clear then from Definition \ref{defn:scr} that the number of columns in $\wt D$ that are interleaving with $c$ is the same before and after a commutation with a column adjacent to $c$.
\end{proof}

We now consider how resolutions fit into the skein theory of grid diagrams. To begin, let us define what we mean by such a resolution. 

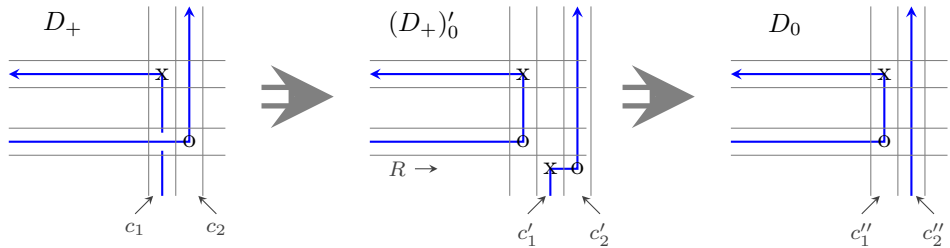
\begin{figure}[ht]
	\[\begin{tikzpicture}[>=stealth,scale=1.2]

	\draw[blue,thick,->,cm={1,0,0,1,(0,0)}]
		(0,0) -- (2,0) -- (2,1.5);
	\draw[blue,thick,cm={1,0,0,1,(0,0)}]
		(1.7,-0.6) -- (1.7,-0.1);
	\draw[blue,thick,->,cm={1,0,0,1,(0,0)}]
		(1.7,0.1) -- (1.7,0.75) -- (0,0.75);
	\draw[blue,thick,->,cm={1,0,0,1,(4,0)}]
		(0,0) -- (1.7,0) -- (1.7,0.75) -- (0,0.75);
	\draw[blue,thick,->,cm={1,0,0,1,(4,0)}]
		(2,-0.6) -- (2,-0.3) -- (2.3,-0.3) -- (2.3,1.5);
	\draw[blue,thick,->,cm={1,0,0,1,(8,0)}]
		(0,0) -- (1.7,0) -- (1.7,0.75) -- (0,0.75);
	\draw[blue,thick,->,cm={1,0,0,1,(8,0)}]
		(2,-0.6) -- (2,1.5);
	\foreach \t in {(0,0),(4,0),(8,0)}
	\foreach \v in {0,0.3,0.6}
		\draw[gray,thin,cm={1,0,0,1,\t}]
		(1.55+\v,-0.6) -- (1.55+\v,1.5);
	\draw[gray,thin,cm={1,0,0,1,(4,0)}]
		(2.45,-0.6) -- (2.45,1.5);
	\foreach \t in {(0,0), (4,0), (8,0)}
		\draw[gray,thin,cm={1,0,0,1,\t}]
			(0,-0.15) -- (2.4,-0.15)
			(0,0.15) -- (2.4,0.15)
			(0,0.6) -- (2.4,0.6)
			(0,0.9) -- (2.4,0.9);
	\foreach \t in {(0,0),(4,0),(8,0)}
		\draw[cm={1,0,0,1,\t}] 
			(1.7,0.75) node {x};
	\foreach \t in {(0,0),(4.3,-0.3)}
		\draw[cm={1,0,0,1,\t}] 
			(2,0) node {o};
	\foreach \t in {(4,0), (8,0)}
		\draw[cm={1,0,0,1,\t}]
			(1.7,0) node {o};
	\draw[cm={1,0,0,1,(4,0)}]
		(2,-0.3) node {x};
	\draw[gray!50!black,->,cm={1,0,0,1,(0,0)}]
		(1.4,-0.8)-- node[at start,below] {{\footnotesize $c_1$}}(1.6,-0.6);
	\draw[gray!50!black,->, cm={1,0,0,1,(0,0)}]	
		(2.3,-0.8) -- node[at start,below] {{\footnotesize $c_2$}}(2.1,-0.6);

	\draw[gray!50!black,->, cm={1,0,0,1,(4,0)}]	
		(2.55,-0.8) -- node[at start,below]{{\footnotesize $c_2'$}}(2.35,-0.6);
	\draw[gray!50!black,->, cm={1,0,0,1,(4,0)}]
		(1.75,-0.8) -- node[at start,below]{{\footnotesize $c_1'$}}(1.95,-0.6);
	
	\draw[gray!50!black,->,cm={1,0,0,1,(4,0)}]
		(0.5,-0.3) -- node[at start,left] {{\footnotesize $R$}} (0.75,-0.3);

	\draw[gray!50!black,->,cm={1,0,0,1,(8,0)}]
		(1.45,-0.8) --node[at start,below] {{\footnotesize $c_1''$}} (1.65,-0.6);
	\draw[gray!50!black,->,cm={1,0,0,1,(8,0)}]
		(2.25,-0.8) --node[at start,below] {{\footnotesize $c_2''$}} (2.05,-0.6);
\foreach \tx in {0.3,4.3}
	\draw[gray,->,line width=9pt] (2.5+\tx,0.5) -- (3.3+\tx,0.5);
\foreach \tx in {0.3,4.3}
	\draw[line width=5pt,white] (2.4+\tx,0.5) -- (2.8+\tx,0.5);

\comment{
	\draw[gray,->,line width=8pt] (2.5,-1.5) -- (3.75,-2.25);
	\draw[line width=4.5pt,white] (2.4,-1.44) -- (3.36,-2.02);
}
	\draw
		(0.6,1.3) node {$D_+$}
		(4.6,1.3) node {$(D_+)_0'$}
		(8.6,1.3) node {$D_0$};

	\end{tikzpicture}\]
\caption{Resolution at a positive skein crossing}
\label{resolvCross}
\end{figure}

Given a grid diagram $D_+$ and a positive skein crossing of $D_+$, choose one of the $\bb O$ markings in the skein crossing. There exists some choice of grid projection that, near this marking, looks like the leftmost projection in Figure \ref{resolvCross} (or a $180^\circ$ rotation of this picture). Let $L_\times$ be the 1-singular link that has one double point at the crossing depicted in this projection, and is identical elsewhere to the link associated to $D_+$. We take the resolution $L_0$ of $L_\times$ defined in (\ref{eqn:resl'ns}).

This operation is local and we see that the link associated to the grid diagram $(D_+)_0'$, pictured in the middle of Figure \ref{resolvCross}, is isotopic to $L_0$. Since columns $c_1'$ and $c_2'$ are adjacent in the diagram $(D_+)_0'$ any commutation of row $R$ will be non-interleaving, and so an isotopy. Therefore, we may commute this row until we have a triple of markings that can be destabilized. After this destabilization we get the grid diagram $D_0$, as shown in Figure \ref{resolvCross}, with an associated link that is isotopic to $L_0$. Note that the grid diagram $D_0$ is obtained from $D_+$ by interchanging the $\theta_1$ coordinates of the $\bb O$ markings in the two columns.

Likewise, given a grid diagram $D_-$, a negative skein crossing of $D_-$, and an $\bb O$ marking in the skein crossing, there is a grid projection of $D_-$ that looks like the leftmost projection in Figure \ref{resolvNegCross}. We then define the resolution of this skein crossing similarly, getting the grid diagram $D_0$ depicted in Figure \ref{resolvNegCross}. In this case, we get the diagram $D_0$ by interchanging the $\bb X$ markings in the columns. Based on our observations, we make the following definition.

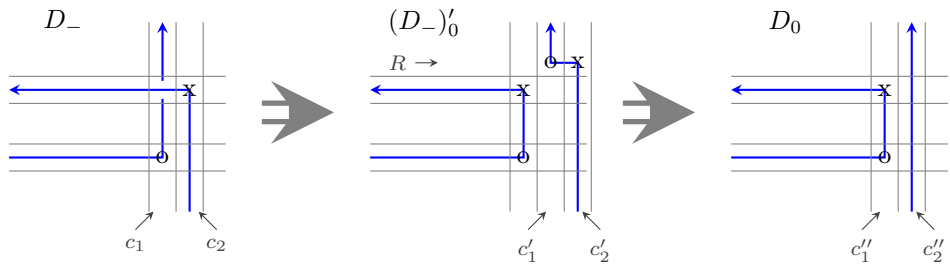
\begin{figure}[ht]
	\[\begin{tikzpicture}[>=stealth,scale=1.2]

	\draw[blue,thick,cm={1,0,0,1,(0,0)}]
		(0,0) -- (1.7,0) -- (1.7,0.65); 
	\draw[blue,thick,->,cm={1,0,0,1,(0,0)}]
		(1.7,0.85) -- (1.7,1.5);
	\draw[blue,thick,->,cm={1,0,0,1,(0,0)}]
		(2,-0.6) -- (2,0.75) -- (0,0.75);
	\draw[blue,thick,->,cm={1,0,0,1,(4,0)}]
		(0,0) -- (1.7,0) -- (1.7,0.75) -- (0,0.75);
	\draw[blue,thick,->,cm={1,0,0,1,(4,0)}]
		(2.3,-0.6) -- (2.3,1.05) -- (2,1.05) -- (2,1.5);
	\draw[blue,thick,->,cm={1,0,0,1,(8,0)}]
		(0,0) -- (1.7,0) -- (1.7,0.75) -- (0,0.75);
	\draw[blue,thick,->,cm={1,0,0,1,(8,0)}]
		(2,-0.6) -- (2,1.5);
	\foreach \t in {(0,0),(4,0),(8,0)}
	\foreach \v in {0,0.3,0.6}
		\draw[gray,thin,cm={1,0,0,1,\t}]
		(1.55+\v,-0.6) -- (1.55+\v,1.5);
	\draw[gray,thin,cm={1,0,0,1,(4,0)}]
		(2.45,-0.6) -- (2.45,1.5);
	\foreach \t in {(0,0), (4,0), (8,0)}
		\draw[gray,thin,cm={1,0,0,1,\t}]
			(0,-0.15) -- (2.4,-0.15)
			(0,0.15) -- (2.4,0.15)
			(0,0.6) -- (2.4,0.6)
			(0,0.9) -- (2.4,0.9);
	\foreach \t in {(0,0),(3.7,0),(7.7,0)}
		\draw[cm={1,0,0,1,\t}] 
			(2,0.75) node {x};
	\foreach \t in {(0,0),(4,0),(8,0)}
		\draw[cm={1,0,0,1,\t}] 
			(1.7,0) node {o};
	\foreach \t in {(4.3,1.05)}
		\draw[cm={1,0,0,1,\t}]
			(1.7,0) node {o};
	\draw[cm={1,0,0,1,(4,0)}]
		(2.3,1.05) node {x};
	\draw[gray!50!black,->,cm={1,0,0,1,(0,0)}]
		(1.4,-0.8)-- node[at start,below] {{\footnotesize $c_1$}}(1.6,-0.6);
	\draw[gray!50!black,->, cm={1,0,0,1,(0,0)}]	
		(2.3,-0.8) -- node[at start,below] {{\footnotesize $c_2$}}(2.1,-0.6);

	\draw[gray!50!black,->, cm={1,0,0,1,(4,0)}]	
		(2.55,-0.8) -- node[at start,below]{{\footnotesize $c_2'$}}(2.35,-0.6);
	\draw[gray!50!black,->, cm={1,0,0,1,(4,0)}]
		(1.75,-0.8) -- node[at start,below]{{\footnotesize $c_1'$}}(1.95,-0.6);
	
	\draw[gray!50!black,->,cm={1,0,0,1,(4,0)}]
		(0.5,1.05) -- node[at start,left] {{\footnotesize $R$}} (0.75,1.05);

	\draw[gray!50!black,->,cm={1,0,0,1,(8,0)}]
		(1.45,-0.8) --node[at start,below] {{\footnotesize $c_1''$}} (1.65,-0.6);
	\draw[gray!50!black,->,cm={1,0,0,1,(8,0)}]
		(2.25,-0.8) --node[at start,below] {{\footnotesize $c_2''$}} (2.05,-0.6);
\foreach \tx in {0.3,4.3}
	\draw[gray,->,line width=9pt] (2.5+\tx,0.5) -- (3.3+\tx,0.5);
\foreach \tx in {0.3,4.3}
	\draw[line width=5pt,white] (2.4+\tx,0.5) -- (2.8+\tx,0.5);

\comment{
	\draw[gray,->,line width=8pt] (2.5,-1.5) -- (3.75,-2.25);
	\draw[line width=4.5pt,white] (2.4,-1.44) -- (3.36,-2.02);
}
	\draw
		(0.6,1.5) node {$D_-$}
		(4.6,1.5) node {$(D_-)_0'$}
		(8.6,1.5) node {$D_0$};

	\end{tikzpicture}\]
\caption{Resolution at a negative skein crossing}
\label{resolvNegCross}
\end{figure}

\begin{defn} Given a grid diagram $D$ and a positive (resp.\ negative) skein crossing of $D$, define $D_0$ to be the grid diagram that differs from $D$ only in the columns of the skein crossing, where it differs only by interchanging the $\theta_1$ coordinates of the $\bb O$ (resp.\ $\bb X$) markings.
\label{defn:gridResln}
\end{defn}

\begin{rem} In Figure \ref{resolvCross} (resp.\ \ref{resolvNegCross}), the grid diagram $D_0$ is obtained by interchanging the $\theta_1$-coordinates of the $\bb O$ (resp.\ $\bb X$) markings in the skein crossing. Note that if we change the $\theta_1$-coordinates of the $\bb X$ (resp.\ $\bb O$) markings instead, the resulting diagram differs from $D_0$ by a non-interleaving commutation of columns, and so its associated link is also isotopic to $L_0$. Finally, note that if $D_+$ and $D_-$ differ only as shown, by a skein crossing change, then the grid diagram $D_0$ is obtained as a resolution of each, so $D_+, D_-$, and $D_0$ form a skein triple. 
\end{rem}

In the course of the next section we will prove the existence of a HOMFLY polynomial of links in $L(p,q)$ (Theorem \ref{mainThm}), and in the proof we will need the function $scr$ to decrease under a resolution $D_{\pm}\leadsto D_0$. To see that $scr$ behaves this way, we use a map $H$ from the columns of a grid diagram to the coordinate plane. The map $H$ is defined below and we point out some of its properties.

Let $\scr C(D)$ be the set of columns of a grid diagram $D$ in $L(1,0)=S^3$ and choose an $\alpha$-curve $\alpha_0$. Define a map $H_{\alpha_0}:\scr C(D)\to\rls^2$ in the following way. Using the orientation on $\vec\beta$ we can define a height function $\theta_2$ on the annulus $T-\alpha_0$. For $c\in\scr C(D)$, let $u$ (resp.\ $v$) be the smaller (resp.\ larger) of the $\theta_2$ coordinates of the markings of $c$. We define $H_{\alpha_0}(c)=(u,v)$. For convenience of notation, we will drop the subscript of $H_{\alpha_0}$ unless we wish to emphasize the dependence of the map on the choice of $\alpha_0$.

By our definition, if $D$ has grid number $n$, then  $H(\scr C(D))$ is a set of $n$ points in the plane between the $y$-axis and the line $y=x$. Further the map only depends on the positions of the markings in a column, so if $rD$ is the grid diagram obtained from $D$ by reversing orientation (interchanging the $\bb X$ and $\bb O$ markings), then the image set $H(rD)$ equals the image $H(D)$. 

\begin{figure}[ht]
\begin{tikzpicture}[>=stealth,scale=0.9]
		\draw[<->,thick,cm={1,0,0,1,(-2,0)}]
			(0,-3.5)--(0,3.5);
		\draw[<->,thick,cm={1,0,0,1,(0,-2)}]
			(-3.5,0)-- (3.5,0);
		\draw[->,thick]
			(-2,-2) -- (3.5,3.5);
		\fill[gray!60!white]
			(-2,0) -- (-2,1.5) -- (0,1.5) -- (0,0);
		\fill[gray!60!white]
			(0,1.5) -- (1.5,1.5) -- (1.5,3.5) -- (0,3.5);
		\draw[thick,dashed]
			(-2, 1.5) -- (1.5,1.5)
			(0,3.5) -- (0,0)
			(1.5,1.5) -- (1.5,3.5)
			(-2,0) -- (0,0);
		\filldraw
			(0,1.5) circle (2pt);
		\draw
			(0,1.5) node [above left] {$p$};
\end{tikzpicture}
\caption{The interleaving regions of $p$, a point in the image of $H:\scr C(D)\to\rls^2$}
\label{fig:intImH}
\end{figure}
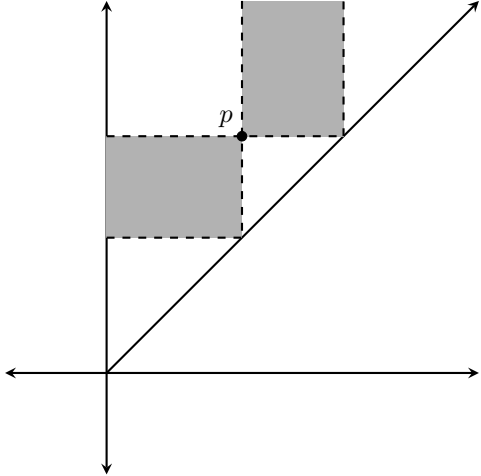

Finally, we point out that if $c$ is a column of $D$ and $p=H(c)$, it can be easily checked that another column $c'$ is interleaving with $c$ if and only if $H(c')$ is in one of the shaded regions of the plane shown in Figure \ref{fig:intImH}. Call these regions the \emph{interleaving regions} of $p$.

\begin{rem}
Observe that the property of two columns being interleaving in a grid diagram for a link in $S^3$ is completely independent of a preferred $\alpha$-curve (Definition \ref{defn:scr}). This implies that, although the map $H$ depends on the choice of $\alpha_0$, the property of $H(c')$ being in one of the interleaving regions of $H(c)$ does not depend on such a choice.
\label{rem:IndepAlpha_0}
\end{rem}


\section{The HOMFLY polynomial in $L(p,q)$}
\label{sec:HOMFLY}
\subsection{Trivial Links in $L(p,q)$.}
\label{subsec:TrivLinks}
Fix a lens space $L(p,q)$ with Heegaard torus $\Sigma$ as described in Section \ref{sec:prelim}. As mentioned in Remark \ref{rpRem}, after a choice of $\alpha_0$ and $\beta_0$ all grid diagrams are given coordinates $\setn{(\theta_1,\theta_2)}{\theta_1\in[0,p), \theta_2\in[0,1)}$.

Given a grid diagram $D$, let $GN(D)$ denote the grid number of $D$. Given a link $K$, set $GN(K)$ to be
	\[GN(K)=\min\setn{GN(D)}{D\text{ is a grid diagram for a link isotopic to }K}.\] 
\begin{figure}[ht]
	\[\begin{tikzpicture}[scale=2,>=stealth]
	\draw[step=0.75,cm={1,0,0,1,(0,0)},very thick]
		(0,0) grid (5.25,0.75);

	\draw[blue,thick]
		(0.375,0.375)--(0.375,0.75)
		(1.125,0.375)--(1.125,0);
	\foreach \x in {0.75,2.25}
	\draw[blue,thick,->]
		(4.875-\x,0)--(4.875-\x,0.375);
	\foreach \x in {0.75,2.25}
	\draw[blue,thick]
		(4.875-\x,0)--(4.875-\x,0.75);
	\draw[blue,thick]
		(0.375,0.375) -- (1.125,0.375);
	\draw
		(0.375,0.375) node {{\bf O}}
		(1.125,0.375) node {{\bf X}};
	\foreach \l in {0,...,6}
	\draw
		(0.75*\l,-0.1) node [below] {\l};
	\foreach \l in {0,...,5}
	\draw
		(0.75*\l+1.5,0.75) node [above] {\l};
	\draw
		(0,0.75) node [above] {5}
		(0.75,0.75) node [above] {6}
		(5.25,-0.1) node [below] {0};
	\end{tikzpicture}\]
\caption{The grid number 1 knot $K_3$ in $L(7,2)$}
\label{fig:GN1knot}
\end{figure}
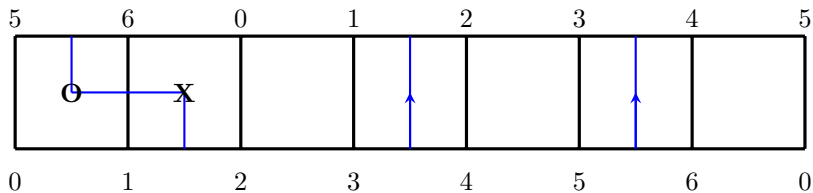

\begin{defn}
A knot $K$ in $L(p,q)$ is called a \emph{trivial knot} if $GN(K)=1$. If $\mu(K)=i$, we write $K=K_i$. Under this notation the unknot described in Remark \ref{rem:unknot} is a trivial knot and written $K_0$.
\label{defn:trivKnot}
\end{defn}

\begin{rem}
Note that $\mu$ is only defined up to congruence class mod $p$. Let $D_{K_i}$ denote the grid diagram with grid number one of the trivial knot $K_i$. Since $p,q$ are coprime, there exists some $K_i$ for each $0\le i\le p-1$. See Figure \ref{fig:GN1knot} for the diagram and a projection of $K_3$ in $L(7,2)$.
\end{rem}
Having defined trivial knots we now define a collection of trivial links so that one component trivial links are trivial knots. Let $\cl I=(m_0,m_1,\ldots,m_{p-1})$ be a $p$-tuple of non-negative integers. We construct a grid diagram $D(\cl I)$ with grid number $n=\sum m_i$ as follows. Place the markings $\bb O=\set{O_1,O_2,\ldots,O_n}$ so that $0<\theta_1(O_i)<1$ for each $1\le i\le n$, and so that
\al{
	\theta_1(O_i)+\frac1n	&= \theta_1(O_{i+1})\qquad\text{and}\\
	\theta_2(O_i)-\frac1n	&= \theta_2(O_{i+1}).
}
for each $1\le i\le n-1$. Note that these conditions force $O_1$ to be in the left-uppermost fundamental parallelogram; i.e. $O_1$ has coordinates $\left(\frac1{2n},1-\frac1{2n}\right)$.

In order to place the markings $\bb X=\set{X_1,X_2,\ldots,X_n}$, we define a bijection $\sigma:\set{0,1,\ldots,p-1}\to\set{0,1,\ldots,p-1}$ by the congruence $\sigma(i) q\equiv i\mod p$. Now, for each $i$ such that $\sum_{k=0}^{l-1}m_{\sigma(k)}+1\le i\le\sum_{k=0}^{l}m_{\sigma(k)}$, with $0\le l\le p-1$, let $v_i$ be the directed vertical arc (in the column of $O_i$) that is oriented away from $O_i$, satisfies $\alpha_0\cdot v_i=\sigma(l)$, and has one endpoint at $O_i$ and the other at the center of a fundamental parallelogram that is in the same row as $O_i$. Note that our convention on the orientation of $\alpha_0$ (see the beginning of Section \ref{sec:hty}) then forces $v_i$ to be oriented upward in figures. Place $X_i$ at the other endpoint of $v_i$. We call the resulting grid diagram $D(\cl I)$. An example is shown in Figure \ref{fig:TrivLink} (the vertical arc $v_2$ in the figure is not part of the grid diagram, but is depicted to help illustrate the definition).

Recall the definition of the components of a grid diagram (Definition \ref{defn:compDiag}). We first note that every component of $D(\cl I)$ has grid number one, since for each $1\le i\le n$, $O_i$ and $X_i$ are in the same row and the same column. Secondly, let $v_i$ be the vertical arc from $O_i$ to $X_i$ as in the construction of $D(\cl I)$. Denote the homology class of the component of $D(\cl I)$ corresponding to $O_i$ by $\mu(O_i)=\alpha_0\cdot v_i$. Then our construction (by the definition of $\sigma$) guarantees that if
	\begin{equation}\mu(O_i)q\text{ (mod }p)<\mu(O_j)q\text{ (mod }p)\label{eqn:orderO}\end{equation}
then $i<j$. Further, the coordinates of the $\bb O$ markings place these markings on a diagonal of a parallelogram in $T-\alpha_0-\beta_0$.

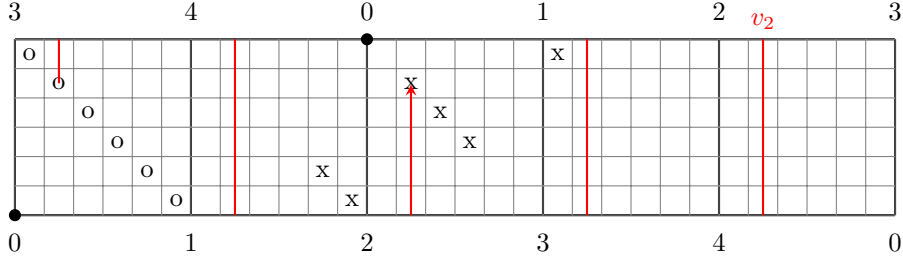
\begin{figure}
	\[\begin{tikzpicture}[scale=0.13,>=stealth]
	\draw[step=18,thick]
		(0,0) grid (90,18);
	\draw[gray,thin,step=3]
		(0,0) grid (90,18);
	\filldraw 
		(0,0) circle (16pt)
		(36,18) circle (16pt);
	\foreach \num in {0,1,2,3}	
	\draw
		(0+\num*18,-1) node[below] {\num}
		(36+\num*18,19) node[above] {\num};
	\draw
		(0,19) node[above] {3}
		(18,19) node[above] {4}
		(72,-1)node[below] {4}
		(90,-1)node[below] {0};
		
	\foreach \x in {(55.5,16.5), (40.5,13.5), (43.5,10.5), (46.5,7.5),(31.5,4.5),(34.5,1.5)}
		\draw 	\x node {x};
	\foreach \o in {(1.5,16.5),(4.5,13.5),(7.5,10.5),(10.5,7.5),(13.5,4.5),(16.5,1.5)}
		\draw	\o node {o};	
	\draw[red,thick,->]
		(4.5,13.5)--(4.5,18) (58.5,0)--(58.5,18) (22.5,0)--(22.5,18) (76.5,0)-- node[at end,above]{$v_2$}(76.5,18) (40.5,0)--(40.5,13.35);
		\end{tikzpicture}\]
	\caption{The trivial link diagram $D(\cl I)$ in $L(5,2)$ with $\cl I=(0,1,2,0,3)$.}
	\label{fig:TrivLink}
	\end{figure}

\begin{defn} For any $p$-tuple $\cl I$ of non-negative integers, define $K(\cl I)$ to be the link associated to $D(\cl I)$. We call a link $K$ in $L(p,q)$ a \emph{trivial link} if $K$ is isotopic to $K(\cl I)$ for some $\cl I$. The $p$-tuple $\cl I$ is called the \emph{index set} associated to $K$ and $D(\cl I)$ is called a \emph{trivial link diagram}. Note that each grid number one knot $K_i$ corresponds to an index set with 1 in the $i^{th}$ position and 0 otherwise.
\label{defn:trivLink}
\end{defn}

In the course of proving Theorem \ref{mainThm} we will need that each trivial link $K$ admits a grid diagram that minimizes the function $scr$ in that free homotopy class of link, among diagrams with minimal grid number (just as $n$ component unlinks minimize crossing number in their homotopy class). It may be that the trivial link diagram for $K$ does not achieve this, however in the following lemma we prove the existence of a grid diagram that does.

Let $\cl I=(m_0,m_1,\ldots,m_{p-1})$, with $m_i\ge0$ for each $i$. Consider all grid diagrams $D$ in $L(p,q)$ such that the components of $D$ have grid number one (they are in the set $\{D_{K_0},D_{K_1},\ldots,D_{K_{p-1}}\}$), and so that $D$ has $m_i$ components of type $D_{K_i}$ for each $0\le i\le p-1$. Call this finite collection of grid diagrams $\scr D(\cl I)$. Let $n=\sum m_i$. Then the grid number of every diagram in $\scr D(\cl I)$ is $n$.

\begin{lem}
Consider the link $K(\cl I)$ in $L(p,q)$ for an index set $\cl I$. There is a grid diagram $\hat D(\cl I)$ which has associated link isotopic to $K(\cl I)$, such that $scr(\hat D(\cl I))$ is the minimum of the image $scr(\scr D(\cl I))\subset\ints_{\ge0}$.
\label{lem:DefinetrivLink}
\end{lem}

\begin{proof}
Definition \ref{defn:scr} gives a map $scr_{\cl I}:\scr D(\cl I)\to\ints_{\ge0}$ by restriction. If $z$ is the minimum of $\im scr_{\cl I}$, let $\scr M(\cl I)=scr_{\cl I}^{-1}(z)$. We must show that there is a diagram in $\scr M(\cl I)$ whose associated link is isotopic to $K(\cl I)$. To do so, we consider $\scr M(\cl I)$ more thoroughly. 

Recall from Proposition \ref{prop:scrProperties} that $scr$ is invariant under column commutation, so if $D\in\scr M(\cl I)$ and we get $D'$ from $D$ by a column commutation, then $D'\in\scr M(\cl I)$. Begin with any grid diagram $D\in\scr M(\cl I)$ and perform column commutations until the $\theta_1$ coordinate of every $\bb O$ marking is between 0 and 1 (this is possible since each component has grid number 1). The resulting grid diagram (which by abuse of notation we still call $D$, and which may have an associated link that is not isotopic to our original choice) is in $\scr M(\cl I)$.

Note that the $\bb O$ markings of $D$ are in one-to-one correspondence with the components and with the columns. Therefore, perhaps after more commutations of columns, we may guarantee that if $O,O'$ are two $\bb O$ markings and if 
	\begin{equation}\mu(O)q\text{ (mod }p)<\mu(O')q\text{ (mod }p)\label{eqn:orderOb}\end{equation}
then $\theta_1(O)<\theta_1(O')$ (here $\mu(O)$ is the homology class of the component of $D$ corresponding to $O$, as in (\ref{eqn:orderO})). As a result, all $\bb O$ markings with a fixed homology type are grouped into consecutive columns. Finally, among the $\bb O$ markings with homology $\mu(O)=i$, perform commutations of columns until the $\theta_2$ coordinates are decreasing (among the markings with $\mu(O)=i$). Do this for each $i$ with $m_i>0$. Denote the resulting grid diagram, which is contained in $\scr M(\cl I)$, by $\hat{D}(\cl I)$.

The process above results in a grid diagram $\hat D(\cl I)$ with an order on the $\bb O$ markings $\set{\hat O_1,\hat O_2,\ldots, \hat O_n}$ of $\hat D(\cl I)$ so that:
\begin{enumerate}
	\item[(a)] $\theta_1(\hat O_i)=\frac{2i-1}{2n}$;
	\item[(b)] if $\mu(\hat O_i)q\text{ (mod }p)<\mu(\hat O_j)q\text{ (mod }p)$ then $i<j$;
	\item[(c)] if $\mu(\hat O_i)=\mu(\hat O_{i+1})$ then $\theta_2(\hat O_i)>\theta_2(\hat O_{i+1})$.
\end{enumerate}

Note that by condition (a), $\theta_1(\hat O_i)=\theta_1(O_i)$ for each $i$. This, in addition to condition (b), implies that for every $1\le i\le n$ we have $\mu(\hat O_i)=\mu(O_i)$. 

Now, suppose the $p$-tuple $\cl I$ is zero in every coordinate except one. Then (c) makes $\hat D(\cl I)=D(\cl I)$, so $D(\cl I)$ is in $\scr M(\cl I)$ and we are finished. Suppose instead there is more than one homology class represented among the components. Property (c) implies that if $\theta_2(\hat O_i)<\theta_2(\hat O_{i+1})$ then $\mu(\hat O_i)\ne\mu(\hat O_{i+1})$. Thus, since $i+1$ is not less than $i$, it must be that 
\begin{equation}
\mu(\hat O_i)q\text{ (mod }p)<\mu(\hat O_{i+1})q\text{ (mod }p).
\label{inequalRowsNI}
\end{equation}
Let $k=\mu(O_{i+1})$ and for any $1\le j\le n$, let $r_j$ be the row of $\hat D(\cl I)$ containing $\hat O_j$ (and thus also $\hat X_j$). Property (a) and inequality (\ref{inequalRowsNI}) then imply that $r_i$ and $r_{i+j}$ are non-interleaving rows for any $1\le j\le m_k$. 

So, for each $0\le k\le p-1$, we may perform row commutations on the $m_k$ rows of $\hat D(\cl I)$ that have components in homology class $k$, one row at a time, until the corresponding $\bb O$ markings are on the diagonal (have coordinates $\left(\frac{2i-1}{2n},1-\frac{2i-1}{2n}\right)$). Each of these row commutations is non-interleaving, using (a), (b) and the fact that the two rows involved in each commutation correspond to different homology classes. These row commutations describe an isotopy starting at the link associated to $\hat D(\cl I)$ and ending at the link associated to $D(\cl I)$. Since $\hat D(\cl I)$ is in $\scr M(\cl I)$, this finishes the proof.
\end{proof}

\subsection{Every link is homotopic to a trivial link}

In this section we show that every link in $L(p,q)$ is homotopic to a unique trivial link. More precisely, we show that this homotopy can be realized by a sequence of particular grid moves (see Proposition \ref{trivlinksProp}).

Note that a trivial link is determined by an index set $\cl I$. Any two different trivial links cannot be in the same homotopy class, since their index sets will differ and so at least one of their components differs in free homotopy class. Thus if a link is homotopic to some trivial link, it is homotopic to a unique trivial link. We will show that there is a particular type of homotopy from any link to a trivial link.

\begin{lem} Let $K$ be a Legendrian link in $L(p,q)$ associated to a grid diagram $D_K$. Suppose $D_K$ has a component with grid number more than $1$. Then there exists a sequence of commutations (both interleaving and non-interleaving) followed by a destabilization giving a new grid diagram $D'$ such that $GN(D')<GN(D_K)$.
\label{lemRuth}
\end{lem}

Note that in Lemma \ref{lemRuth} we arrive at the grid diagram $D'$ without using the column exchange described in Remark \ref{rem:illegalcommut}. We avoid the use of such moves for contact geometric considerations. Before proving Lemma \ref{lemRuth} we use it to find a homotopy from any link to a trivial link through commutations and destabilizations.
 
\begin{prop}
Let $D_K$ be a grid diagram for the link $K\subset L(p,q)$. There is a sequence of commutations and destabilizations taking $D_K$ to a trivial link diagram. This provides a homotopy from $K$ to a trivial link.
\label{trivlinksProp}
\end{prop}
\begin{proof}
We prove Lemma \ref{lemRuth} below. By that lemma there is a grid diagram $D_{K'}$ corresponding to a link $K'$ such that the grid number of each component of $D_{K'}$ equals 1, and a sequence of commutations and destabilizations takes $D_K$ to $D_{K'}$. Since each component of $D_{K'}$ has grid number 1, we may use skein crossings to successively interchange adjacent columns of $D_{K'}$ until we arrive at a diagram $\hat D$ with its $\bb O$ markings satisfying conditions (a), (b), and (c) from the proof of Lemma \ref{lem:DefinetrivLink}. By that proof, the link associated to $\hat D$ is a trivial link. This process is a sequence of column commutations, either interleaving or non-interleaving. Since non-interleaving commutations and destabilizations correspond to an isotopy, Lemma \ref{trivlinksLem} implies that the resulting sequence of diagrams, from $D_K$ to $\hat D$, corresponds to a homotopy from $K$ to a trivial link.
\end{proof}

\begin{proof} [Proof of Lemma \ref{lemRuth}]
We introduce the \emph{length} of arcs in a grid projection. Every grid projection $P$ is a union of $GN(P)$ horizontal arcs and $GN(P)$ vertical arcs. Define the \emph{length} of an arc by the number of fundamental parallelograms it traverses. For example, the vertical arc in Figure \ref{fig:GN1knot} has length 3 and the horizontal arc has length 1. 

Note that if two markings $X_i$ and $O_j$ share a row (resp.\ a column), there are two possible arcs that could join them in a grid projection. Define $len(X_iO_j)$ to be the minimum of the lengths of these two arcs. Note that if $len(X_iO_j)=1$ then $X_i$ and $O_j$ are in adjacent columns (resp.\ rows). The converse, however, is not true since one column of the grid diagram (a component of $T-\vec\beta$) has $p$ fundamental parallelograms in the same row. 

For clarity, in the following figures we depict our argument by schematics rather than depicting the link projections on a grid diagram. The schematics show only the markings and horizontal and vertical arcs of the projection, but we remind the reader that{--}just as in a grid projection{--}a vertical arc in the schematic could be intersecting $\alpha_0$ multiple times
. The arrows in the schematics indicate how commutations will affect the columns, and a circled triple of markings indicates a destabilization.

Continuing with the proof, let $\hat D$ be a component of $D_K$ with $GN(\hat D)>1$. Then there are three markings of $\hat D$, say $X_1, O_1, X_2$, where $X_1, X_2$ are distinct, $X_1, O_1$ are in the same column, and $X_2,O_1$ are in the same row. Furthermore, among all such triples of markings of $D_K$, choose a triple $X_1,O_1,X_2$ such that $len(X_1O_1)$ is minimal.

Denote the column containing $X_1$ and $O_1$ by $c_1$ and the row containing $O_1$ and $X_2$ by $r_1$. Define $O_2$ to be the $\bb O$ marking in the column of $X_2$ and call this column $c_2$. Let $X_3$ be the $\bb X$ marking in the same row as $O_2$, which row we call $r_2$ (see Figure \ref{figRuth2}). 

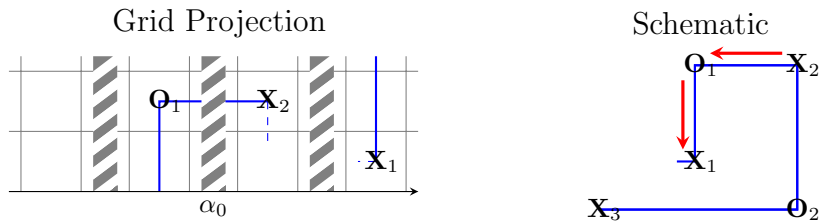
\begin{figure}[ht]
	\[\begin{tikzpicture}[scale=0.8,>=stealth]

	\draw
		(3.3,2.8) node {\begin{Large}Grid Projection\end{Large}};
		
	\foreach \c in {0,...,3}
	\draw[step=1,gray,thin,cm={1,0,0,1,(1.8*\c,0)}]
		(-0.2,0) grid (1.2,2.25);
	\draw[blue,thick]
		(4.1,1.5) -- (2.3,1.5) -- (2.3,0)
		(5.9,2.25) -- (5.9,0.5);
	\draw[blue,dashed]
		(4.1,1.5) -- (4.1,0.8)
		(5.9,0.5) -- (5.6,0.5);
\foreach \ty in {0,0.4,0.8,1.2,1.6}
	\foreach \b in {1.2,3,4.8}
	\fill[gray,cm={1,0,0,1,(0,\ty)}]
		(\b,0) -- (\b,0.2) -- (\b+0.4,0.5) -- (\b+0.4,0.3) -- cycle;
\foreach \ty in {0.2,0.6,1,1.4}
	\foreach \b in {1.2,3,4.8}
	\fill[white,cm={1,0,0,1,(0,\ty)}]
		(\b,0) -- (\b,0.2) -- (\b+0.4,0.5) -- (\b+0.4,0.3) -- cycle;
	\foreach \b in {1.2,3,4.8}
	\fill[gray,cm={1,0,0,1,(0,0)}]
		(\b,2.25) -- (\b+0.33,2.25) -- (\b,2) -- cycle;
\foreach \tty in {0,-4,-8,-12}
\foreach \tx in {0,8}
	\foreach \b in {1.2,3,4.8}
	\fill[gray,cm={1,0,0,1,(0,0)}]
		(\b+0.4,0) -- (\b+0.4,0.1) -- (\b+0.277,0) -- cycle;

	\draw[->]
		(-0.2,0) -- node [below] {$\alpha_0$} (6.6,0);

	\draw
		(2.4,1.5) node {{\bf O$_1$}}
		(4.2,1.5) node {{\bf X$_2$}}
		(6,0.5) node {{\bf X$_1$}};

	\draw
		(11.3,2.8) node {\begin{Large}Schematic\end{Large}};
		
	\foreach \g in {0,1.6}
	\draw[blue,thick]
		(11.2,0.5) -- (11.2,2.1) -- (12.9,2.1) -- (12.9,-0.3) -- (9.6,-0.3);
	\draw[thick,blue]
		(11.2,0.5) -- (10.9,0.5);
	\draw
		(11.3,0.5) node {{\bf X$_1$}}
		(11.3,2.1) node {{\bf O$_1$}}
		(13,2.1) node {{\bf X$_2$}};	
	\draw
		(13,-0.3) node {{\bf O$_2$}}
		(9.7,-0.3) node {{\bf X$_3$}};
	\draw[red,very thick,<-]
		(11,0.7) -- (11,1.85);
	\draw[red,very thick,->]
		(12.65,2.3) -- (11.45,2.3);
	\end{tikzpicture}\]
\caption{Decreasing Grid Number}
\label{figRuth2}
\end{figure}

Since $X_1$ and $X_2$ are distinct, we know the columns containing the pairs $(X_1,O_1)$ and $(X_2,O_2)$ are distinct. Thus there is a sequence of commutations of columns of $D_K$ after which we have one of two cases. In case 1, $len(O_1X_2)=1$. In case 2, which is depicted in Figures \ref{figRuth3} and \ref{figRuth4}, $len(O_2X_3)=1$. Each commutation (either interleaving or non-interleaving) exchanges column $c_2$ with an adjacent column, the result being that $len(O_1X_2)$ decreases. As previously shown, a commutation corresponds to either a skein crossing change or a Legendrian isotopy. 

We now consider the two cases.

{\bf Case 1:} By the minimality of our choice of $len(X_1O_1)$ we can commute row $r_1$ with an adjacent row, decreasing $len(X_1O_1)$ until $len(X_1O_1)=1$. Since $len(O_1X_2)=1$, $c_1$ and $c_2$ are adjacent columns, so each of the row commutations is non-interleaving. Performing a destabilization on $(X_1,O_1,X_2)$, the grid number drops by one.

\begin{figure}[ht]
	\[\begin{tikzpicture}[scale=0.7,>=stealth]
	
	\foreach \tx in {0,7}
	\draw[blue,thick] (9.6-\tx,0.5) -- (9.6-\tx,2.1) -- (11.8-\tx,2.1);
	\draw[blue,thick] (16.6,0.5) -- (16.6,2.1) -- (18.2,2.1) -- (18.2,3);
	\foreach \tx in {-7,0,7}
	\draw[thick,blue] (9.6-\tx,0.5) -- (9.3-\tx,0.5);
	\foreach \tx in {0,7}
	\draw[thick,blue] (11.8-\tx,2.1) -- (11.8-\tx,1.5-1.8*\tx/7) -- (11.2-\tx,1.5-1.8*\tx/7);
	\draw[thick,blue] (11.2,1.5) -- (11.2,1.95) (11.2,2.25) -- (11.2,3);
	\draw[thick,blue] (4.2,-0.3) -- (4.2,1.95) (4.2,2.25)--(4.2,3);
	\foreach \tx in {-7,0,7}
	\draw 
		(9.7-\tx,0.5) node {{\bf X$_1$}}
		(9.7-\tx,2.1) node {{\bf O$_1$}};
	\foreach \tx in {0,7}
	\draw
		(11.9-\tx,2.1) node {{\bf X$_2$}};
	\draw
		(18.3,2.17) node{{\bf X$'$}};
	\foreach \tx in {0,7}
	\draw
		(11.9-\tx,1.5-1.8*\tx/7) node {{\bf O$_2$}}
		(11.3-\tx,1.5-1.8*\tx/7) node {{\bf X$_3$}};
	\foreach \tx in {-7,0,7}
	\draw
		(11.3-\tx,3) node {{\bf O$_3$}};

	\draw[red,thick,->] (5,0.1) -- (5,1.7);
	\draw [red,thick] (11.55,1.8) node[circle,draw,minimum size=30]{};
\foreach \tx in {0,7}
	\draw[gray,->,line width=10pt] (6.3+\tx,1.5) -- (8.2+\tx,1.5);
\foreach \tx in {0,7}
	\draw[line width=6pt,white] (6.2+\tx,1.5) -- (7.3+\tx,1.5);
		
	\end{tikzpicture}\]
\caption{Decreasing Grid Number:Case 2a}
\label{figRuth3}
\end{figure}
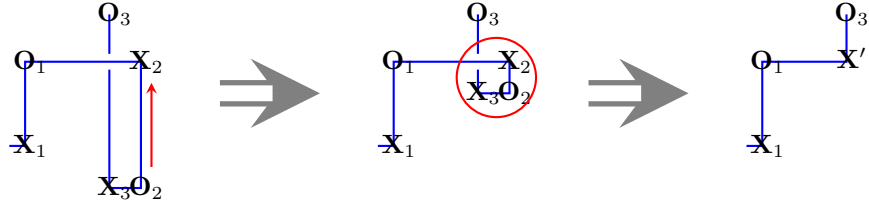

{\bf Case 2:} We are in one of the two situations of Figures \ref{figRuth3} and \ref{figRuth4}. Let $O_3$ be the marking that shares a column with $X_3$. 

In case 2a, $len(X_3O_3)\ge len(X_2O_2)$, so we can commute row $r_2$ until $len(X_2O_2)=1$, with all commutations being non-interleaving (since $len(X_3O_2)=1$). This is possible since $X_3$ is not the same marking as $X_2$, for this would force $GN(\hat D)=1$. Note that $X_3$ may not be distinct from $X_1$. Performing a destabilization on the triple $(X_2, O_2,X_3)$, the grid number decreases by one. 

Otherwise, we are in case 2b where $len(X_3O_3)<len(X_2O_2)$. Commute row $r_2$, using only non-interleaving commutations, until $len(X_3O_3)=1$, which again is possible since $GN(\hat D)\ne1$ and $len(X_3O_2)=1$. Perform a destabilization on the triple $(O_2,X_3,O_3)$. Again the destabilization drops the grid number of the diagram by one.

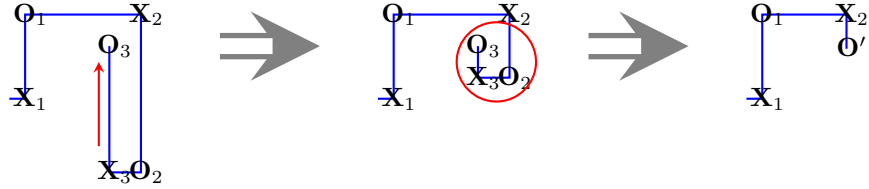
\begin{figure}[ht]
	\[\begin{tikzpicture}[scale=0.7,>=stealth]
	
	\foreach \tx in {0,7}
	\draw[blue,thick] (9.6-\tx,0.5) -- (9.6-\tx,2.1) -- (11.8-\tx,2.1);
	\draw[blue,thick] (16.6,0.5) -- (16.6,2.1) -- (18.2,2.1) -- (18.2,1.45);
	\foreach \tx in {-7,0,7}
	\draw[thick,blue] (9.6-\tx,0.5) -- (9.3-\tx,0.5);
	\foreach \tx in {0,7}
	\draw[thick,blue] (11.8-\tx,2.1) -- (11.8-\tx,0.9-1.8*\tx/7) -- (11.2-\tx,0.9-1.8*\tx/7);
	\draw[thick,blue] (11.2,0.9) -- (11.2,1.5);
	\draw[thick,blue] (4.2,-0.9) -- (4.2,1.5);
	\foreach \tx in {-7,0,7}
	\draw 
		(9.7-\tx,0.5) node {{\bf X$_1$}}
		(9.7-\tx,2.1) node {{\bf O$_1$}};
	\foreach \tx in {0,7}
	\draw
		(11.9-\tx,2.1) node {{\bf X$_2$}};
	\draw
		(18.3,2.1) node{{\bf X$_2$}};
	\foreach \tx in {0,7}
	\draw
		(11.9-\tx,0.9-1.8*\tx/7) node {{\bf O$_2$}}
		(11.3-\tx,0.9-1.8*\tx/7) node {{\bf X$_3$}};
	\foreach \tx in {0,7}
	\draw
		(11.3-\tx,1.5) node {{\bf O$_3$}};
	\draw (18.3,1.5) node {{\bf O$'$}};

	\draw[red,thick,->] (4,-0.4) -- (4,1.2);
	\draw [red,thick] (11.55,1.2) node[circle,draw,minimum size=30]{};
\foreach \tx in {0,7}
	\draw[gray,->,line width=10pt] (6.3+\tx,1.5) -- (8.2+\tx,1.5);
\foreach \tx in {0,7}
	\draw[line width=6pt,white] (6.2+\tx,1.5) -- (7.3+\tx,1.5);
		
	\end{tikzpicture}\]
\caption{Decreasing Grid Number:Case 2b}
\label{figRuth4}
\end{figure}

\end{proof}

\begin{rem}
Examining the proof of Lemma \ref{lemRuth}, we see that no \underline{row} commutation performed in the reduction is interleaving.
\label{noRowCross}
\end{rem}

\subsection{The polynomial invariant}

As a consequence of Proposition \ref{trivlinksProp}, there is a unique trivial link in each homotopy class of links. In this section we define a complexity function on grid diagrams so that trivial link diagrams minimize the complexity function. In classical knot theory one often considers the unlinking number of a link diagram as part of a complexity function. We consider an analogous number for grid diagrams.

\begin{defn}
Let $D$ be a grid diagram for $K$. Define $u(D)$ to be the minimum number of skein crossing changes in a sequence of commutations and destabilizations that take $D$ to a trivial link diagram $D(\cl I)$, where the only interleaving commutations are column commutations. Proposition \ref{trivlinksProp} , Remark \ref{noRowCross}, and the proof of Lemma \ref{lem:DefinetrivLink} imply that there is such a number and that if $u(D)=0$ then the link associated to $D$ is isotopic to a trivial link.
\label{defn:unlkNum}
\end{defn}

{\bf Theorem \ref{mainThm}}.\quad
{\it Let $\scr L$ be the set of isotopy classes of links in $L(p,q)$ and let $\scr{TL}\subset\scr L$ denote the set of isotopy classes of trivial links. Define $\scr{TL}^*\subset\scr{TL}$ to be those trivial links with no nullhomotopic components. Let $U$ be the isotopy class of the standard unknot, a local knot in $L(p,q)$ that bounds an embedded disk. Suppose we are given a value $J_{p,q}(T)\in\ints[a^{\pm1},z^{\pm1}]$ for every $T\in\scr{TL}^*$. Then there is a unique map $J_{p,q}:\scr{L}\to\ints[a^{\pm1},z^{\pm1}]$ such that
	\en{
	\item[(i)] $J_{p,q}$ satisfies the skein relation
		\begin{equation*}
		a^{-p}J_{p,q}(L_+)-a^pJ_{p,q}(L_-)=zJ_{p,q}(L_0).
		\end{equation*}
	\item[(ii)] $J_{p,q}(U)=a^{-p+1}$.
	\item[(iii)] $J_{p,q}\left(U\coprod L\right)=\frac{a^{-p}-a^p}{z}J_{p,q}(L)$.
	}
}
\medskip

\begin{rem}
The map $J_{p,q}$ is defined to be the map $J_M$, with $M=L(p,q)$, given by Theorem \ref{thm:psInvt} except that we define $a$ such that $a^p=v$. While not essential to the proof of the theorem, the change of variable is useful for a future application: in \cite{C2} the author uses a particular normalization of $J_{p,q}$ in order to achieve an analogue of the Franks-Williams-Morton inequality. To prove the theorem at hand, by Theorem \ref{thm:psInvt} we only need show that $J_{p,q}$ is Laurent-polynomial valued. Note that we have not specified what the fixed normalization on the trivial links is. Indeed, such a choice is not canonical. 
\label{rem:normalization}
\end{rem}

\begin{proof}
Define a function $\psi$ which assigns to every grid diagram $D$ in $L(p,q)$ the triple of numbers \[\psi(D)=(GN(D),scr(D),u(D)).\] Ordering these triples lexicographically, we will use $\psi$ as a complexity function.

There is a grid diagram $\hat D$ (the diagram $\hat D(\cl I)$ of Lemma \ref{lem:DefinetrivLink}) associated to any trivial link $T$ so that if $D_K$ is any grid diagram with associated link $K$, and $K$ is in the same homotopy class as $T$, then $\psi(\hat D)\le\psi(D_K)$, equality holding only if $T$ is isotopic to $K$. For in such a case, $GN(\hat D)\le GN(D_K)$ since $\hat D\in\scr D(\cl I)$, with $\cl I$ the index set. If $GN(\hat D)=GN(D_K)$ then $D_K\in\scr D(\cl I)$, but by definition $\hat D\in\scr M(\cl I)$, implying $scr(\hat D)\le scr(D_K)$. As $u(\hat D)=0\le u(D_K)$ by definition, with equality only if $K$ is isotopic to $T$, we see that $\psi(\hat D)<\psi(D_K)$ if $K$ is not a trivial link.

Our argument inducts on the complexity function $\psi$. Let $K$ be a non-trivial link in $L(p,q)$ and $D$ a grid diagram associated to $K$. There is a sequence of $u(D)$ skein crossing changes that, with non-interleaving commutations and destabilizations, takes $D$ to the trivial link diagram in its homotopy class. Suppose the first skein crossing in this sequence is positive and let $D=D_+$, $D_-$, and $D_0$ be diagrams as described in Section \ref{sec:skTheory} (differing only at the two columns of the skein crossing) with associated links $K=K_+, K_-$, and $K_0$. By the skein relation,
	\begin{equation}
	a^{-p}J_{p,q}(K)-a^pJ_{p,q}(K_-)=zJ_{p,q}(K_0).
	\label{eqn:skRel1}
	\end{equation}
But $GN(D_-)=GN(D)$, and $scr(D_-)=scr(D)$ since the skein crossing change is a commutation of columns. As our choice of crossing change makes $u(D_-)<u(D)$, we have that $\psi(D_-)<\psi(D)$.

Recall the defintion of the grid diagram $D_0$ in Definition \ref{defn:gridResln}. $D_0$ has the same grid number as $D$. We prove below that 

\begin{lem}
$scr(D_0)<scr(D)$.
\label{lem:scrDrop}
\end{lem}

Provided Lemma \ref{lem:scrDrop} is true, $D_-$ and $D_0$ have strictly lower complexity than $D$, so by induction we may assume that $J_{p,q}(K_-)$ and $J_{p,q}(K_0)$ are Laurent polynomials. The skein relation (\ref{eqn:skRel1}) then implies $J_{p,q}(K)\in\ints[a^{\pm1},z^{\pm1}]$ as well. The argument is similar if the first skein crossing change in the sequence is negative, only we use
	\[a^{-p}J_{p,q}(K_+)-a^pJ_{p,q}(K)=zJ_{p,q}(K_0).\]

\end{proof}

\begin{proof}[Proof of Lemma \ref{lem:scrDrop}]
The proof will proceed as follows. Recall our definition of a map $H$ from the columns of a grid diagram in $S^3$ to a subset of the coordinate plane (see the end of Section \ref{sec:hty}). Also recall that $scr(D)$ is defined to be the number of pairs of interleaving columns of $\wt D$, the lift of $D$ to $S^3$, and these pairs need not be adjacent. 

Given a skein crossing of $D$ which is resolved to give $D_0$, the lifts $\wt D$ and $\wt{D_0}$ differ in exactly $2p$ columns, each of these being a lift of a column in the given skein crossing of $D$. Both $\wt D$ and $\wt{D_0}$ have grid number $p\cdot GN(D)$. Let $c$ be the column of $\wt D$ between the $\beta$ curves $\beta_i, \beta_{i+1}$. Then define $c'$ to be the column in $\wt{D_0}$ between $\beta_i$ and $\beta_{i+1}$. If $c$ is not one of the $2p$ columns covering the skein crossing, we call $c$ \emph{unaltered} and otherwise $c$ is called \emph{altered}. Note that a column $c$ in $\wt D$ is unaltered exactly when $H(c)=H(c')$. So a pair of columns $\set{c,d}$ that are both unaltered in $\wt D$ is interleaving if and only if $\set{c',d'}$ is interleaving in $\wt{D_0}$.

To prove the lemma then, we show that the number of interleaving pairs of columns $\set{c, d}$ in $\wt D$ such that one of $c$ or $d$ is altered, is larger than the number of interleaving pairs $\set{c', d'}$ in $\wt{D_0}$. This is shown by proving inequalities (\ref{eqn:unaltered}) and (\ref{eqn:altered}) below. Now to the proof.

For one of the two columns in the given skein crossing of $D$, choose $c_1$ as one of the $p$ lifts of that column in $\wt D$. Choose $c_2$ to be the lift of the other column in the skein crossing so that $c_2$ is adjacent to $c_1$. As in the outline of the proof, let $c_1'$ and $c_2'$ be the corresponding columns in $\wt{D_0}$. 

Since $\wt D$ and $\wt{D_0}$ differ only in a specified way, we can refer to the same curve $\alpha_0$ in both diagrams. For $H=H_{\alpha_0}:\scr C(\wt D)\to\rls^2$, let $p_i=H(c_i)$ and for $H=H_{\alpha_0}:\scr C(\wt{D_0})\to\rls^2$, let $p_i'=H(c_i')$. If $q\in\im H$ and $c$ is some column of $\wt D$ (or $\wt{D_0}$) then define $q(c)=1$ if $H(c)$ is in an interleaving region of $q$ (recall, the interleaving regions of a point in the image of $H$ are as shown in Figure \ref{fig:intImH}). Define $q(c)=0$ otherwise. We will also abuse notation and sometimes write $q(H(c))$ to mean $q(c)$.

To prove Lemma \ref{lem:scrDrop} we first show that for any unaltered column $c$, 
	\begin{equation}
	p_1(c)+p_2(c)\ge p_1'(c')+p_2'(c').
	\label{eqn:unaltered}
	\end{equation}
Second we show that if $d_1,d_2$ is a pair of adjacent altered columns in $\wt D$, then 
	\begin{equation}
	\sum_{i=1}^2(p_1(d_i)+p_2(d_i))\ge\sum_{i=1}^2(p_1'(d_i')+p_2'(d_i')).
	\label{eqn:altered}
	\end{equation}
Since $c_1$ and $c_2$ interleave each other in $\wt D$ but $c_1'$ and $c_2'$ are non-interleaving in $\wt{D_0}$ by design, this would show that $scr(D_0)\le scr(D)-p$.

To show that (\ref{eqn:unaltered}) holds, let $c$ be an unaltered column of $\wt D$ and let $q=H(c)=H(c')$. Also let $Z_1, Z_2$ be the markings in $c_1, c_2$ that will swap columns to give $c_1',c_2'$. We can choose $\alpha_0$ so that $\theta_2(Z_1)$ and $\theta_2(Z_2)$ are the $x$-coordinates of $H(c_1)$ and $H(c_2)$ respectively. Without loss of generality, we may also assume that $\theta_2(Z_1)<\theta_2(Z_2)$ and that $\alpha_0$ is on the boundary of the row containing $Z_1$. Hence, no point in $\im H$ has a smaller $x$-coordinate than $p_1=H(c_1)$. In fact, with such a choice of $\alpha_0$ the points $p_1, p_2$ appear as in Figure \ref{fig2:intImH}, and we get $p_1', p_2'$ by interchanging the $x$-coordinates of $p_1,p_2$.

\begin{figure}[ht]
\begin{tikzpicture}[>=stealth,scale=0.8]
		\draw[<->,thick,cm={1,0,0,1,(-2,0)}]
			(0,-3.5)--(0,3.5);
		\draw[<->,thick,cm={1,0,0,1,(0,-2)}]
			(-3.5,0)-- (3.5,0);
		\draw[->,thick]
			(-2,-2) -- (3.5,3.5);
		\fill[gray!60!white]
			(-2,-1.9) -- (-2,1) -- (-1.9,1) -- (-1.9,-1.9);
		\fill[gray!60!white]
			(-1.9,1) -- (1,1) -- (1,3.5) -- (-1.9,3.5);
		\fill[pattern = north east lines, pattern color = black!80!white]
			(-2,2.5) -- (-0.5,2.5) -- (-0.5,-0.5) -- (-2,-0.5);
		\fill[pattern = north east lines, pattern color = black!80!white]
			(-0.5,3.5) -- (-0.5,2.5) -- (2.5,2.5) -- (2.5,3.5);
		\draw[thick,dashed]
			(-2, 1) -- (1,1)
			(-1.9,3.5) -- (-1.9,-1.9)
			(1,1) -- (1,3.5)
			(-2,-1.9) -- (-1.9,-1.9);
		\filldraw
			(-1.9,1) circle (2pt)
			(-0.5,2.5) circle (2pt);
		\draw
			(-1.9,1) node [below left] {$p_1$}
			(-0.5,2.5) node [above left] {$p_2$};
		\filldraw
			(-1,2) circle (2pt);
		\draw[->,thick]
			(-2.2,2) -- node [at start, left] {$q$} (-1.1,2);
		\draw
			(1,-1) node {$H(\scr C(\wt D))$};

		\draw[<->,thick,cm={1,0,0,1,(-2,0)}]
			(8,-3.5)--(8,3.5);
		\draw[<->,thick,cm={1,0,0,1,(0,-2)}]
			(4.5,0)-- (11.5,0);
		\draw[->,thick]
			(6,-2) -- (11.5,3.5);
		\fill[gray!60!white]
			(6,1) -- (6,-0.5) -- (7.5,-0.5) -- (7.5,1);
		\fill[gray!60!white]
			(9,3.5) -- (9,1) -- (7.5,1) -- (7.5,3.5);
		\fill[pattern = north east lines, pattern color = black!80!white]
			(6,2.5) -- (6.1,2.5) -- (6.1,-1.9) -- (6,-1.9);
		\fill[pattern = north east lines, pattern color = black!80!white]
			(6.1,3.5) -- (6.1,2.5) -- (10.5,2.5) -- (10.5,3.5);

		\draw[thick,dashed]
			(6, 1) -- (9,1)
			(9,3.5) -- (9,1)
			(7.5,-0.5) -- (7.5,3.5)
			(6,-0.5) -- (7.5,-0.5);
		\filldraw
			(7.5,1) circle (2pt)
			(6.1,2.5) circle (2pt);
		\draw
			(7.5,1) node [below left] {$p_1'$}
			(6.1,2.5) node [above left] {$p_2'$};
		\filldraw
			(7,2) circle (2pt);
		\draw
			(7,2) node [below left] {$q$};
		\draw
			(9,-1) node {$H(\scr C(\wt{D_0}))$};

\end{tikzpicture}
\caption{The interleaving of unaltered columns in $\wt D$}
\label{fig2:intImH}
\end{figure}
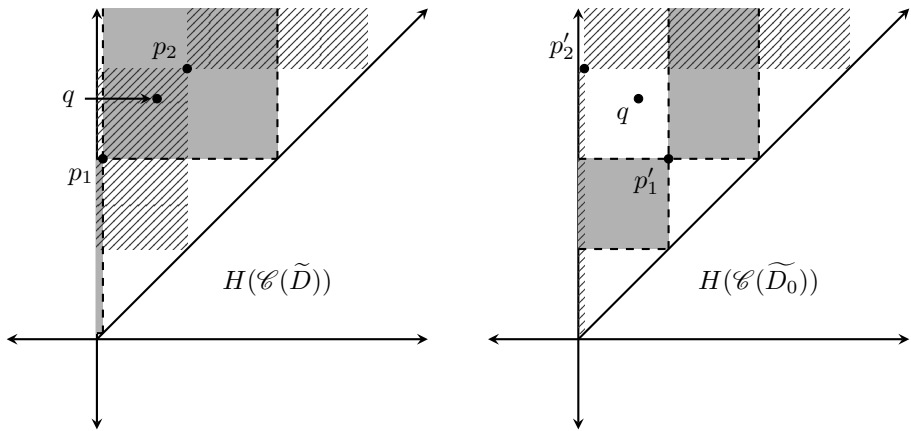

If $q$ is in any region of the plane other than where it appears in Figure \ref{fig2:intImH} then it is evident from the figure that $p_1(c)+p_2(c) = p_1'(c')+p_2'(c')$. On the other hand, if $q$ is as in Figure \ref{fig2:intImH} then $p_1(c)+p_2(c) = 2$ and $p_1'(c')+p_2'(c') = 0$. This shows that (\ref{eqn:unaltered}) holds.

We now show that (\ref{eqn:altered}) holds. As observed in Remark \ref{rem:IndepAlpha_0} our choice of $\alpha_0$ does not change the count in (\ref{eqn:altered}). Given a pair $d_1,d_2$ of interleaving columns in $\wt D$, we say that $\alpha_0$ is \emph{good} for the pair if the $x$-coordinates of $H_{\alpha_0}(d_1)$ and $H_{\alpha_0}(d_2)$ either both come from heights of $\bb O$ markings or both come from heights of $\bb X$ markings. Every pair of interleaving columns has at least 2 good $\alpha$-curves. 

Suppose $d_1$ and $d_2$ are a pair of ajacent, altered columns, $d_1'$ and $d_2'$ are the corresponding columns in $\wt{D_0}$, $q_i=H(d_i)$ and $q_i'=H(d_i')$. If $\alpha_0$ is good for $d_1, d_2$ then the points $q_1, q_1', q_2, q_2'$ are the corners of a rectangle in $\rls^2$, with sides parallel to the axes and with $q_1$ (resp.\ $q_2$) being the lower left (resp.\ upper right) corner. In our proof of (\ref{eqn:unaltered}), the choice of $\alpha_0$ was good for $c_1,c_2$.

\begin{figure}[ht]
\begin{tikzpicture}[>=stealth,scale=0.8]
		\draw[<->,thick,cm={1,0,0,1,(-2,0)}]
			(0,-3.5)--(0,3.5);
		\draw[<->,thick,cm={1,0,0,1,(0,-2)}]
			(-3.5,0)-- (3.5,0);
		\draw[->,thick]
			(-2,-2) -- (3.5,3.5);
		\fill[gray!60!white]
			(-2,0) -- (-2,1.5) -- (0,1.5) -- (0,0);
		\fill[gray!60!white]
			(0,1.5) -- (1.5,1.5) -- (1.5,3.5) -- (0,3.5);
		\fill[pattern = north east lines, pattern color = black!80!white]
			(1,3.5) -- (1,3) -- (3,3) -- (3,3.5)
			(-2,3) -- (1,3) -- (1,1) -- (-2,1);
		\draw[thick,dashed]
			(-2, 1.5) -- (1.5,1.5)
			(0,3.5) -- (0,0)
			(1.5,1.5) -- (1.5,3.5)
			(-2,0) -- (0,0);
		\filldraw
			(0,1.5) circle (2pt)
			(1,3) circle (2pt);
		\filldraw[thick]
			(1.25,2.5) circle (2pt)
			(2,2.5) circle (2pt)
			(2,3.25) circle (2pt)
			(1.25,3.25) circle (2pt);
		\draw[thick]
			(1.25,2.5) -- (2,2.5) -- (2,3.25) -- (1.25,3.25) -- cycle;
		\draw
			(0,1.5) node [above left] {$p_1$}
			(1,3) node [above left] {$p_2$};
		\draw
			(1,-1) node {$H(\scr C(\wt D))$};

		\draw[<->,thick,cm={1,0,0,1,(6,0)}]
			(0,-3.5)--(0,3.5);
		\draw[<->,thick,cm={1,0,0,1,(8,-2)}]
			(-3.5,0)-- (3.5,0);
		\draw[->,thick]
			(6,-2) -- (11.5,3.5);
		\fill[gray!60!white]
			(6,1) -- (6,1.5) -- (9,1.5) -- (9,1);
		\fill[gray!60!white]
			(9,1.5) -- (9.5,1.5) -- (9.5,3.5) -- (9,3.5);
		\fill[pattern = north east lines, pattern color = black!80!white]
			(8,3.5) -- (8,3) -- (11,3) -- (11,3.5)
			(6,3) -- (8,3) -- (8,0) -- (6,0);
		\draw[thick,dashed]
			(6, 1.5) -- (9.5,1.5)
			(9,3.5) -- (9,1)
			(9.5,1.5) -- (9.5,3.5)
			(6,1) -- (9,1);
		\filldraw
			(9,1.5) circle (2pt)
			(8,3) circle (2pt);
		\draw
			(9,1.5) node [above left] {$p_1'$}
			(8,3) node [above left] {$p_2'$};
		\draw
			(9,-1) node {$H(\scr C(\wt{D_0}))$};

\end{tikzpicture}
\caption{The interleaving of altered columns in $\wt D$ with a common good choice of $\alpha_0$}
\label{fig3:intImH}
\end{figure}
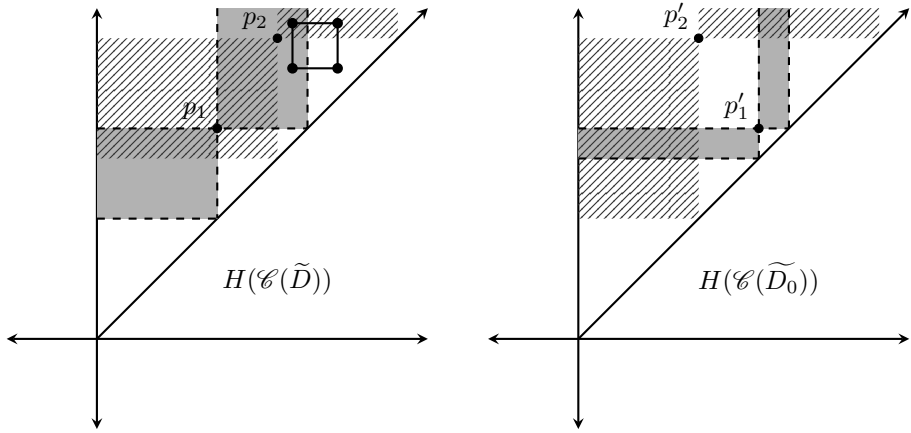

Now suppose there is a choice of $\alpha_0$ that is good for both pairs $c_1, c_2$ and $d_1, d_2$ simultaneously. With this choice of $\alpha_0$ the points $p_i, p_i'$ for $i=1,2$ appear as in Figure \ref{fig3:intImH} with corresponding interleaving regions. Note that for an arbitrary point $a$ in the second octant, we have that $p_1(a)+p_2(a)\ge p_1'(a)+p_2'(a)$. Let $R$ be a rectangle in this octant (the sides of $R$ parallel to the axes) with corner points in the image of $H$, labeled counter-clockwise by $r_1,s_1,r_2,s_2$ with $r_1$ in the lower left corner. If 
 	\begin{equation}
	 \sum_{i=1}^2p_1(r_i)+p_2(r_i)\ge \sum_{i=1}^2p_1(s_i)+p_2(s_i)
	 \label{eqn:rectSum}
	 \end{equation}
then since $p_1(s_j)+p_2(s_j)\ge p_1'(s_j)+p_2'(s_j)$ we see that (\ref{eqn:altered}) is satisfied. One easily checks that (\ref{eqn:rectSum}) is satisfied for any rectangle in the second octant. For the example rectangle shown on the left of Figure \ref{fig3:intImH} we have $\sum p_1(r_i)+p_2(r_i) = 2$ and $\sum p_1(s_i)+p_2(s_i) = 2$.

We are left only to check the case when there is no choice of $\alpha_0$ that is good for both $c_1,c_2$ and $d_1,d_2$. This can only occur when the set of good $\alpha$-curves for $d_1,d_2$ is contained in the set of $\alpha$-curves that are not good for $c_1,c_2$. In such a situation, we have two cases. Either $d_i$ is interleaving with both $c_1$ and $c_2$ for $i=1$ and $i=2$, or $d_i$ is non-interleaving with both $c_1$ and $c_2$, for $i=1$ and $i=2$. In the first case, since the right hand side of (\ref{eqn:altered}) can be at most 4, and the left hand side equals 4, (\ref{eqn:altered}) is satisfied.

\begin{figure}[ht]
\begin{tikzpicture}[>=stealth,scale=0.8]
		\draw[<->,thick,cm={1,0,0,1,(-2,0)}]
			(0,-3.5)--(0,3.5);
		\draw[<->,thick,cm={1,0,0,1,(0,-2)}]
			(-3.5,0)-- (3.5,0);
		\draw[->,thick]
			(-2,-2) -- (3.5,3.5);
		\fill[gray!60!white]
			(-2,-1.9) -- (-2,1) -- (-1.9,1) -- (-1.9,-1.9);
		\fill[gray!60!white]
			(-1.9,1) -- (1,1) -- (1,3.5) -- (-1.9,3.5);
		\fill[pattern = north east lines, pattern color = black!80!white]
			(-2,2.5) -- (-0.5,2.5) -- (-0.5,-0.5) -- (-2,-0.5);
		\fill[pattern = north east lines, pattern color = black!80!white]
			(-0.5,3.5) -- (-0.5,2.5) -- (2.5,2.5) -- (2.5,3.5);
		\draw[thick,dashed]
			(-2, 1) -- (1,1)
			(-1.9,3.5) -- (-1.9,-1.9)
			(1,1) -- (1,3.5)
			(-2,-1.9) -- (-1.9,-1.9);
		\filldraw
			(-1.9,1) circle (2pt)
			(-0.5,2.5) circle (2pt);
		\draw
			(-1.9,1) node [below left] {$p_1$}
			(-0.5,2.5) node [above left] {$p_2$};
		\filldraw
			(1.2,2) circle (2pt)
			(1.7,2.4) circle (2pt)
			(1.2,1.7) circle (2pt)
			(2, 2.4) circle (2pt);
		\draw[->]
			(2,0.9) -- node[at start,right] {$q_{1+i}'$} (1.3,1.6);
		\draw[->]
			(2.8,1.6) -- node[at start,right] {$q_{2+i}'$} (2.1,2.3); 
		\draw[->]
			(0.5,2) -- node [at start,left] {$q_1$} (1.1,2);
		\draw[->]
			(1.7,3.6) -- node [at start, above] {$q_2$} (1.7,2.5);

		\draw[<->,thick,cm={1,0,0,1,(6,0)}]
			(0,-3.5)--(0,3.5);
		\draw[<->,thick,cm={1,0,0,1,(8,-2)}]
			(-3.5,0)-- (3.5,0);
		\draw[->,thick]
			(6,-2) -- (11.5,3.5);
		\fill[gray!60!white]
			(6,-1.9) -- (6,1) -- (6.1,1) -- (6.1,-1.9);
		\fill[gray!60!white]
			(6.1,1) -- (9,1) -- (9,3.5) -- (6.1,3.5);
		\fill[pattern = north east lines, pattern color = black!80!white]
			(6,2.5) -- (7.5,2.5) -- (7.5,-0.5) -- (6,-0.5);
		\fill[pattern = north east lines, pattern color = black!80!white]
			(7.5,3.5) -- (7.5,2.5) -- (10.5,2.5) -- (10.5,3.5);
		\draw[thick,dashed]
			(6, 1) -- (9,1)
			(6.1,3.5) -- (6.1,-1.9)
			(9,1) -- (9,3.5)
			(6,-1.9) -- (6.1,-1.9);
		\filldraw
			(6.1,1) circle (2pt)
			(7.5,2.5) circle (2pt);
		\draw
			(6.1,1) node [below left] {$p_1$}
			(7.5,2.5) node [above left] {$p_2$};
		\filldraw
			(6.3,-1) circle (2pt)
			(6.6,-0.7) circle (2pt)
			(6.3,-1.4) circle (2pt)
			(7, -0.7) circle (2pt);
		\draw[->]
			(7.1,-1.4) -- node[at start,right] {$q_{1+i}'$} (6.4,-1.4);
		\draw[->]
			(7.8,-0.7) -- node[at start,right] {$q_{2+i}'$} (7.1,-0.7); 
		\draw[->]
			(5.5,-1.5) -- node [at start,left] {$q_1$} (6.25,-1.1);
		\draw[->]
			(5.5,-0.7) -- node [at start, left] {$q_2$} (6.4,-0.7);

\end{tikzpicture}
\caption{The interleaving of altered columns in $\wt D$ without a common good choice of $\alpha_0$}
\label{fig4:intImH}
\end{figure}
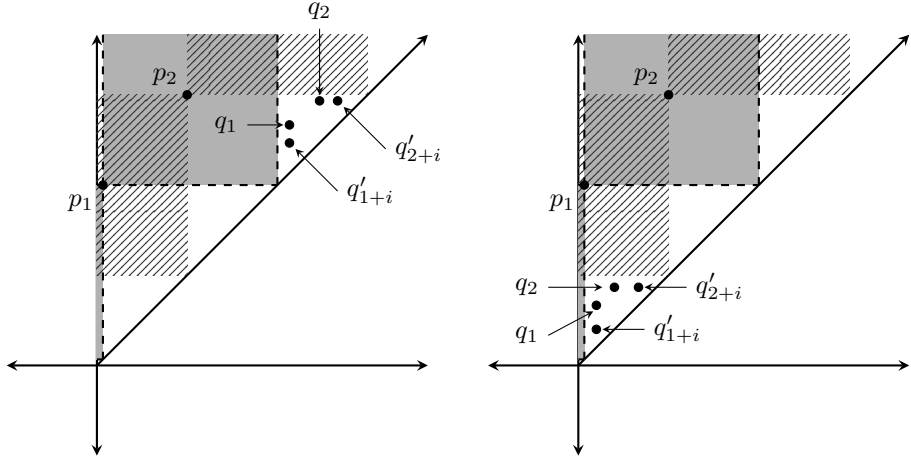

Consider the case when $d_i$ is non-interleaving with both $c_1$ and $c_2$, for $i=1,2$. Then, since $d_1$ is interleaving with $d_2$, the heights of all the markings in $d_1,d_2$ are either between the $\bb O$ markings of $c_1,c_2$ or between the $\bb X$ markings of $c_1,c_2$. Choose $\alpha_0$ so that the $x$-coordinate of $p_1$ is minimal (as we did in Figure \ref{fig2:intImH}). With this choice we have the following: for $i\in\set{1,2}$, both $x$- and $y$-coordinates of $q_i=H(d_i)$ are either numbers between the $y$-coordinates of $p_1$ and $p_2$ or are between the $x$-coordinates of $p_1$ and $p_2$. The first case is depicted on the left of Figure \ref{fig4:intImH} and the second case on the right, and in the figure $i\in\set{1,2}$, and indices are taken mod 2. In both cases, the left-hand and right-hand sides of (\ref{eqn:altered}) are both zero. \end{proof}

\section{Examples and Computations}
\label{sec:comput}

In this section we compute the polynomial $J_{p,q}(L)$ for two examples in $L(5,1)$. We choose our normalization as follows. Recall that grid projections are oriented, with each vertical arc oriented from the $\bb O$ marking to the $\bb X$ marking, and each horizontal arc oriented from the marking in $\bb X$ to that in $\bb O$. Given a trivial link $T$ with index set $\cl I$ and trivial link diagram $D(\cl I)$, we consider the grid projection $P(\cl I)$ associated to $D(\cl I)$ which has the property that for any component $C$ in the projection, $\alpha_0\cdot C\ge0$ and $C\cdot\beta_0\ge0$ with equality when $C$ is nullhomotopic. Since each component of $D(\cl I)$ has only one column and one row, $P(\cl I)$ is determined by $T$. Define $w(P(\cl I))$ to be the writhe of $P(\cl I)$ and define $\mu=\alpha_0\cdot P(\cl I)$ and $\lambda=P(\cl I)\cdot\beta_0$. Finally, let $s(T)=w(P(\cl I))-\frac{\mu\lambda+(\mu-\lambda)}p$, which is well-defined since $T$ determines $P(\cl I)$. Then define $J_{p,q}(T)=a^{ps(T)+1}$. This choice of normalization receives its motivation from Legendrian knot theory. 

The first example we compute is a nullhomotopic knot $B$ which is depicted via a grid diagram at the top of Figure \ref{fig:example1}. The second example $L$, also in $L(5,1)$ and shown in Figure \ref{fig:example2}, is homotopically nontrivial, having $\mu(L)= 2\mod 5$. After computing the two examples we comment how the second example gives rise to a family of examples that are of interest in \cite{C2}.

\begin{exmp} In Figure \ref{fig:example1} we show the skein tree of $B$. Note that the relevant skein crossing in $B$ is positive. The skein crossing change (the left branch in the tree) takes $B$ to the unknot $K_0$. Resolution of the skein crossing (the right branch) gives a link $B_0$.

\begin{figure}[ht]
	\[\begin{tikzpicture}[scale=0.1,>=stealth]
	\draw
		(15,15) node [above] {$B$};
	\draw[step=10,thick]
			(0,0) grid (50,10);
	\draw[gray,thin,step=5]
			(0,0) grid (50,10);	
	\foreach \l in {0,1,2,3,4}
	\draw 
		(10*\l,0) node [below] {\l};
	\draw	
		(50,0) node [below] {0};	
	\foreach \l in {0,1,2,3,4}
	\draw
		(10*\l+10,10) node[above] {\l};
	\draw
		(0,10) node[above] {4};
	\draw[blue,thick]
		(42.5,7.5) -- (50,7.5) (0,7.5) -- (7.5,7.5) -- (7.5,3) (7.5,2) -- (7.5,0)
		(17.5,10) -- (17.5,2.5) -- (2.5,2.5) -- (2.5,7) (2.5,8) -- (2.5,10)
		(42.5,0) -- (42.5,7.5);
	\foreach \x in {(42.5,7.5), (17.5,2.5)}
		\draw 	\x node {x};
	\foreach \o in {(7.5,7.5), (2.5,2.5)}
		\draw	\o node {o};
		
	\draw[dashed,thick]
		(15,-5) -- (5,-20)
		(35,-5) -- (45,-20);
	\draw[cm={1,0,0,1,(-30,-35)}]
		(15,15) node [above] {$B_-=K_0$};
	\draw[step=10,thick,cm={1,0,0,1,(-30,-35)}]
			(0,0) grid (50,10);
	\draw[gray,thin,step=5,cm={1,0,0,1,(-30,-35)}]
			(0,0) grid (50,10);	
	\foreach \l in {0,1,2,3,4}
	\draw[cm={1,0,0,1,(-30,-35)}] 
		(10*\l,0) node [below] {\l};
	\draw[cm={1,0,0,1,(-30,-35)}]	
		(50,0) node [below] {0};	
	\foreach \l in {0,1,2,3,4}
	\draw[cm={1,0,0,1,(-30,-35)}]
		(10*\l+10,10) node[above] {\l};
	\draw[cm={1,0,0,1,(-30,-35)}]
		(0,10) node[above] {4};
	\draw[blue,thick,cm={1,0,0,1,(-30,-35)}]
		(47.5,7.5) -- (50,7.5) (0,7.5) -- (2.5,7.5) -- (2.5,0)
		(12.5,10) -- (12.5,2.5) -- (7.5,2.5) -- (7.5,10)
		(47.5,0) -- (47.5,7.5);
	\foreach \x in {(47.5,7.5), (12.5,2.5)}
		\draw[cm={1,0,0,1,(-30,-35)}]
		 	\x node {x};
	\foreach \o in {(2.5,7.5), (7.5,2.5)}
		\draw[cm={1,0,0,1,(-30,-35)}]
			\o node {o};
	\draw[cm={1,0,0,1,(30,-35)}]
		(35,15) node [above] {$B_0$};
	\draw[step=10,thick,cm={1,0,0,1,(30,-35)}]
			(0,0) grid (50,10);
	\draw[gray,thin,step=5,cm={1,0,0,1,(30,-35)}]
			(0,0) grid (50,10);	
	\foreach \l in {0,1,2,3,4}
	\draw[cm={1,0,0,1,(30,-35)}] 
		(10*\l,0) node [below] {\l};
	\draw[cm={1,0,0,1,(30,-35)}]	
		(50,0) node [below] {0};	
	\foreach \l in {0,1,2,3,4}
	\draw[cm={1,0,0,1,(30,-35)}]
		(10*\l+10,10) node[above] {\l};
	\draw[cm={1,0,0,1,(30,-35)}]
		(0,10) node[above] {4};
	\draw[blue,thick,cm={1,0,0,1,(30,-35)},->]
		(17.5,10) -- (17.5,2.5) -- (7.5,2.5) -- (7.5,0);
	\draw[blue,thick,cm={1,0,0,1,(30,-35)},->]	
		(42.5,0) -- (42.5,7.5) -- (50,7.5);
	\draw[blue,thick,cm={1,0,0,1,(30,-35)}]
		(0,7.5) -- (2.5,7.5) -- (2.5,10); 
	\foreach \x in {(42.5,7.5), (17.5,2.5)}
		\draw[cm={1,0,0,1,(30,-35)}]
		 	\x node {x};
	\foreach \o in {(2.5,7.5), (7.5,2.5)}
		\draw[cm={1,0,0,1,(30,-35)}]
			\o node {o};

	\end{tikzpicture}\]
	\caption{Example: A link $B$ in $L(5,1)$ and its skein tree.}
	\label{fig:example1}
\end{figure}
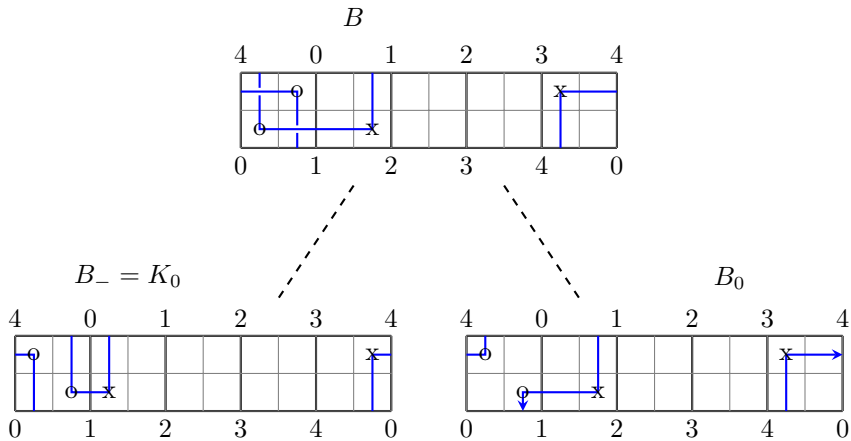

To see that $B_0$ is isotopic to a trivial link, label the $\bb O$ markings $O_1, O_2$ so that $\theta_1(O_1)<\theta_1(O_2)$. The grid projection of $B_0$ in Figure \ref{fig:example1} is such that the vertical arc $v_1$, with endpoint $O_1$, has $\alpha_0\cdot v_1=1>0$. However, the other vertical arc $v_2$, with endpoint $O_2$ is such that $\alpha_0\cdot v_2=-1<0$. So we have $\mu(O_1)=1$ and $\mu(O_2)=4\equiv-1 (\text{ mod }5)$. Since $1\cdot 1 < 4\cdot 1 (\text{ mod }5)$, we see that the grid diagram shown for $B_0$ is the exactly the trivial link diagram $D(0,1,0,0,1)$.

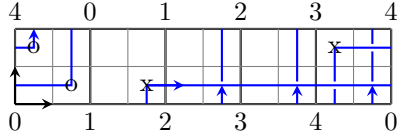
\begin{figure}[ht]
	\[\begin{tikzpicture}[scale=0.1,>=stealth]
	\draw[step=10,thick,cm={1,0,0,1,(30,-35)}]
			(0,0) grid (50,10);
	\draw[gray,thin,step=5,cm={1,0,0,1,(30,-35)}]
			(0,0) grid (50,10);	
	\foreach \l in {0,1,2,3,4}
	\draw[cm={1,0,0,1,(30,-35)}] 
		(10*\l,0) node [below] {\l};
	\draw[cm={1,0,0,1,(30,-35)}]	
		(50,0) node [below] {0};	
	\foreach \l in {0,1,2,3,4}
	\draw[cm={1,0,0,1,(30,-35)}]
		(10*\l+10,10) node[above] {\l};
	\draw[cm={1,0,0,1,(30,-35)}]
		(0,10) node[above] {4};
	\draw[blue,thick,cm={1,0,0,1,(30,-35)}]
		(17.5,0) -- (17.5,2.5);	
	\draw[blue,thick,cm={1,0,0,1,(30,-35)},->]
		(17.5,2.5) -- (22.5,2.5);
	\draw[blue,thick,cm={1,0,0,1,(30,-35)}]
		(17.5,2.5) -- (50,2.5) (0,2.5) -- (7.5,2.5) -- (7.5,10);
	\foreach \x in {47.5, 37.5, 27.5}	
	\draw[blue,thick,cm={1,0,0,1,(30,-35)},->]
		(\x,0) -- (\x,2.5);
	\foreach \x in {37.5, 27.5}
	\draw[blue,thick,cm={1,0,0,1,(30,-35)}]
		(\x,3) -- (\x,10);
	\draw[blue,thick,cm={1,0,0,1,(30,-35)}]
		(47.5,3) -- (47.5,7) (47.5,8) -- (47.5,10);
	\draw[blue,thick,cm={1,0,0,1,(30,-35)}]	
		(42.5,0) -- (42.5,2) (42.5,3) -- (42.5,7.5) -- (50,7.5); 
	\draw[blue,thick,cm={1,0,0,1,(30,-35)},->]
		(0,7.5) -- (2.5,7.5) -- (2.5,10);
	\foreach \x in {(42.5,7.5), (17.5,2.5)}
		\draw[cm={1,0,0,1,(30,-35)}]
		 	\x node {x};
	\foreach \o in {(2.5,7.5), (7.5,2.5)}
		\draw[cm={1,0,0,1,(30,-35)}]
			\o node {o};
	\draw[thick,cm={1,0,0,1,(30,-35)},->]
		(0,0) -- (5,0);
	\draw[thick,cm={1,0,0,1,(30,-35)},->]
		(0,0) -- (0,5);

	\end{tikzpicture}\]
	\caption{The grid projection $P(0,1,0,0,1)$ for the trivial link $B_0$ (shown in blue). For the purposes of calculating $s(B_0)$, recall that $\alpha_0$ is the circle on the torus represented by the bottom (and top) edge, oriented as indicated. The circle $\beta_0$ is the union of the thick vertical black segments in the figure, oriented upward as shown.}
	\label{fig:2ndexample1}
\end{figure}
While $B_0$ is a trivial link with index set $(0,1,0,0,1)$, the corresponding grid projection shown in Figure \ref{fig:example1} is not $P(0,1,0,0,1)$. The projection $P=P(0,1,0,0,1)$ is illustrated in Figure \ref{fig:2ndexample1}, where it is clear that $\alpha_0\cdot P=1+4=5$ and $P\cdot\beta_0 = 1+4 = 5$. As the writhe is $w(P)=5$, we see that $s(B_0)=0$, and so $J_{5,1}(B_0)=a^{0+1}=a$. Therefore,
	\al{
	J_{5,1}(B) 	&= a^{10}J_{5,1}(K_0)+a^5zJ_{5,1}(B_0)\\
				&= a^{10}\left(a^{-4}\right)+a^5z\left(a\right)\\
				&= a^6(1+z).
	}
\end{exmp}

\begin{figure}[ht]
	\[\begin{tikzpicture}[scale=0.085]
	\draw
		(17,17) node [above] {$L$};
	\draw[step=12,thick]
			(0,0) grid (60,12);
	\draw[gray,thin,step=4]
			(0,0) grid (60,12);	
	\foreach \l in {0,1,2,3,4}
	\draw 
		(12*\l,0) node [below] {\l};
	\draw	
		(60,0) node [below] {0};	
	\foreach \l in {0,1,2,3,4}
	\draw
		(12*\l+12,12) node[above] {\l};
	\draw
		(0,12) node[above] {4};
	\draw[blue,thick]
		(58,0) -- (58,2) -- (50,2) -- (50,5) (50,7) -- (50,10) -- (6,10) -- (6,12)
		(54,0) -- (54,1) (54,3) -- (54,6) -- (10,6) -- (10,9) (10,11) -- (10,12);
	\foreach \x in {(50,10), (54,6), (58,2)}
		\draw 	\x node {x};
	\foreach \o in {(6,10),(10,6),(50,2)}
		\draw	\o node {o};
		
	\draw[dashed,thick]
		(15,-5) -- (5,-20)
		(40,-5) -- (50,-20);
	\draw[cm={1,0,0,1,(-35,-40)}]
		(17,17) node [above] {$L_+$};
	\draw[step=12,thick,cm={1,0,0,1,(-35,-40)}]
			(0,0) grid (60,12);
	\draw[gray,thin,step=4,cm={1,0,0,1,(-35,-40)}]
			(0,0) grid (60,12);	
	\foreach \l in {0,1,2,3,4}
	\draw[cm={1,0,0,1,(-35,-40)}] 
		(12*\l,0) node [below] {\l};
	\draw[cm={1,0,0,1,(-35,-40)}]	
		(60,0) node [below] {0};	
	\foreach \l in {0,1,2,3,4}
	\draw[cm={1,0,0,1,(-35,-40)}]
		(12*\l+12,12) node[above] {\l};
	\draw[cm={1,0,0,1,(-35,-40)}]
		(0,12) node[above] {4};
	\draw[blue,thick,cm={1,0,0,1,(-35,-40)}]
		(58,0) -- (58,2) -- (54,2) -- (54,10) -- (2,10) -- (2,12)
		(50,0) -- (50,6) -- (10,6) -- (10,9) (10,11) -- (10,12);
	\foreach \x in {(54,10), (50,6), (58,2)}
		\draw[cm={1,0,0,1,(-35,-40)}]
		 	\x node {x};
	\foreach \o in {(2,10),(10,6),(54,2)}
		\draw[cm={1,0,0,1,(-35,-40)}]
			\o node {o};
	\draw[cm={1,0,0,1,(35,-40)}]
		(37,17) node [above] {$L_0$};
	\draw[step=12,thick,cm={1,0,0,1,(35,-40)}]
			(0,0) grid (60,12);
	\draw[gray,thin,step=4,cm={1,0,0,1,(35,-40)}]
			(0,0) grid (60,12);	
	\foreach \l in {0,1,2,3,4}
	\draw[cm={1,0,0,1,(35,-40)}]
		(12*\l,0) node [below] {\l};
	\draw[cm={1,0,0,1,(35,-40)}]
		(60,0) node [below] {0};	
	\foreach \l in {0,1,2,3,4}
	\draw[cm={1,0,0,1,(35,-40)}]
		(12*\l+12,12) node[above] {\l};
	\draw[cm={1,0,0,1,(35,-40)}]
		(0,12) node[above] {4};
	\draw[blue,thick,cm={1,0,0,1,(35,-40)}]
		(58,0) -- (58,2) -- (54,2) -- (54,6) -- (10,6) -- (10,9) (10,11) -- (10,12)
		(50,0) -- (50,5) (50,7) -- (50,10) -- (2,10) -- (2,12);
	\foreach \x in {(50,10), (54,6), (58,2)}
		\draw[cm={1,0,0,1,(35,-40)}]
		 	\x node {x};
	\foreach \o in {(2,10),(10,6),(54,2)}
		\draw [cm={1,0,0,1,(35,-40)}]
			\o node {o};

	\end{tikzpicture}\]
	\caption{Example: A link $L$ in $L(5,1)$ and its skein tree.}
	\label{fig:example2}
\end{figure}
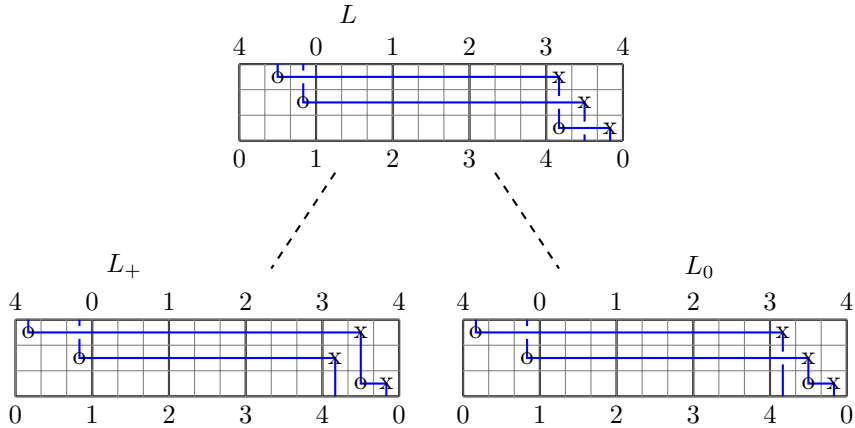

\begin{exmp} In the second example we have a negative skein crossing associated to the first and second columns of the grid diagram for $L$. Via grid moves that correspond to isotopy (a non-interleaving row commutation, destabilization, a non-interleaving column commutation, then another destabilization) the diagram shown for $L_+$ in Figure \ref{fig:example2} is taken to a grid number one diagram of the trivial knot $K_2$. So $L_+$ is isotopic to the trivial knot $K_2$.

The diagram shown in the figure of $L_0$ (after a destabilization) has two grid number one components: both of type $D_{K_1}$. The $\bb O$ markings are adjacent and on a negative slope, so we have the diagram $D(0,2,0,0,0)$ and $L_0$ is trivial.

Since $s(K_2)=\frac15$ and $s(L_0)=-\frac45$, we compute
	\al{
	J_{5,1}(L) &= a^{-10}J_{5,1}(K_2)-a^{-5}zJ_{5,1}(L_0)\\
			&= a^{-10}(a^2) - a^{-5}z(a^{-3})\\
			&= a^{-8}(1-z).
	}
\label{exmp2}
\end{exmp}
\vspace*{-12pt}

Consider the family of links $\set{L_n}_{n\ge0}$ where $L_n$ is the link associated to the grid diagram in Figure \ref{fig:L_n} with grid number $n+2$. Then $L_1$ is the knot $L$ in Example \ref{exmp2} above and $L_0$ is isotopic to the trivial 2-component link in the skein tree of $L$. The link $L_n$ is a knot if $n$ is odd and a 2-component link if $n$ is even. 

\begin{figure}[ht]
	\[\begin{tikzpicture}[scale=0.09]
	\draw[step=28,thick]
			(0,0) grid (140,12.5);
	\draw[step=28,thick,cm={1,0,0,1,(0,4)}]
			(0,11.5) grid (140,28);
	\draw[gray,thin,step=4]
			(0,0) grid (140,12.5)
			(0,15.5) grid (140,32);	
	\foreach \l in {0,1,2,3,4}
	\draw 
		(28*\l,0) node [below] {\l};
	\draw	
		(140,0) node [below] {0};	
	\foreach \l in {0,1,2,3,4}
	\draw
		(28*\l+28,32) node[above] {\l};
	\draw
		(0,32) node[above] {4};
	\foreach \x in {(114,30),(118,26),(130,10), (134,6), (138,2)}
		\draw 	\x node {x};
	\foreach \o in {(22,30),(26,26),(114,22),(118,18),(130,2)}
		\draw	\o node {o};
	
	\draw
		(122.5,32) node[above] {$\overbrace{\qquad\qquad}$}
		(122.5,35) node[above] {\small $n$ columns};
	\fill[white]
		(121,-2) -- (121,33) -- (127,33) -- (127,-2) -- cycle;
	\draw
		(124,14) node {$\ddots$};
					
	\end{tikzpicture}\]
	\caption{Grid diagram associated to the link $L_n$.}
	\label{fig:L_n}
\end{figure}
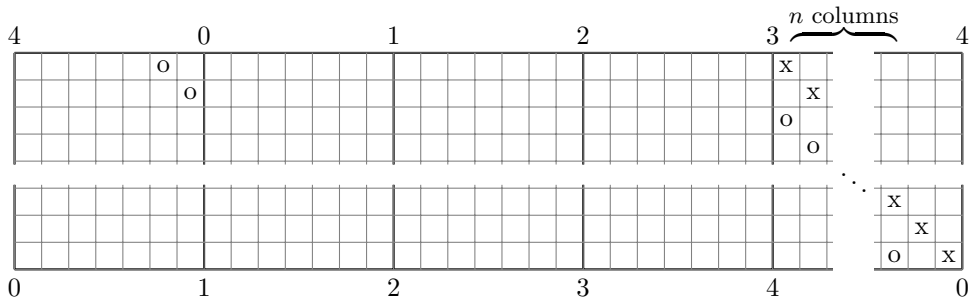

The first two columns of this grid diagram make a negative skein crossing. It is not difficult to see that for $n\ge2$, commutation of these columns of $L_n$ gives a link isotopic to $L_{n-2}$ and the resolution of the same columns gives a link isotopic to $L_{n-1}$. Therefore
	\[J_{5,1}(L_n) = a^{-10}J_{5,1}(L_{n-2})-a^{-5}zJ_{5,1}(L_{n-1}).\]
We know that $J_{5,1}(L_0)=a^{-3}$ and $J_{5,1}(L_1)=a^{-8}(1-z)$ by our computations above. Define $f_n$ recursively: let
	\al{
		f_0= 1,\qquad	&\qquad f_1=1-z,
	}
	and define $f_n=f_{n-2}-zf_{n-1}$ for $n\ge2$. Then the skein relation implies that $J_{5,1}(L_n)=a^{-5n-3}f_n$. Note that the recursive definition of $f_n$ implies that it is not zero for any $n$. This gives a family of links in $L(5,1)$ with the lower degree in $a$ of the HOMFLY polynomial being arbitrarily negative. We will see in \cite{C2} that the corresponding Legendrian links to these diagrams maximize the classical contact invariant, $\tb_\rat+\abs{\rot_\rat}$, in their link type.

\bibliography{HOMpolyrefs}
\bibliographystyle{amsplain}
	
\end{document}